%% file: reflection_working_notes_4.tex
\documentclass[10pt,a4]{amsart}
\usepackage{framed,amssymb}
\usepackage[left=40mm,right=40mm,top=40mm,bottom=40mm]{geometry}
\usepackage{color}
\usepackage[T1]{fontenc}
\usepackage{ae,aecompl}
\usepackage{tikz}
\usepackage{comment}
\usepackage{amsthm,amsmath}
\usepackage{hyperref}
\usepackage{standalone}

\usepackage[outline]{contour}
\contourlength{4pt}
\usetikzlibrary{decorations.markings}

\newcommand{\defeq}{\stackrel{{\text{def}}}{=}} 
\newcommand{\pp}[2]{\frac{\partial #1}{\partial #2}} 
\newcommand{\ppp}[2]{\frac{\partial^2 #1}{\partial #2^2}} 
\newcommand{\pppp}[3]{\frac{\partial^2 #1}{\partial #2 \partial #3}} 
\newcommand{\jump}[1]{\left[#1\right]} 
\newcommand{\Landau}{\mathcal{O}} 
\newcommand{\red}[1]{\color{red} #1 \color{black}}
\newcommand{\cin}{{c_{\textrm{in}}}} 
\newcommand{\cout}{{c_{\textrm{out}}}} 
\newcommand{\lot}{\textrm{LOT}} 

\newtheorem{proposition}{Proposition}

\newtheorem{theorem}{Theorem}

\begin{document}

\title{Shock Reflection in Plane Symmetry}
\date\today
\author[A.~Lisibach]{Andr\'e Lisibach}
\address{Andr\'e Lisibach\\Bern University of Applied Sciences}
\email{andre.lisibach@bfh.ch}
\begin{abstract}
  We consider the problem of shock reflection on a solid wall in plane symmetry for a barotropic fluid. We establish a local in time solution after the point of reflection, thereby determining the state behind the reflected shock. The location of the reflected shock in space time being unknown the problem constitutes a free boundary problem. Furthermore, the quantity describing the thermodynamic state of the fluid along the wall being also unknown, complete data for the problem is only given in a single point, the reflection point. 
\end{abstract}

\maketitle
\tableofcontents

\section{Introduction}
The reflection of a shock on a solid wall is a classic problem in gas dynamics. We look at the situation in plane symmetry, i.e.~the quantities describing the fluid depend only on one space variable. A typical application is gas flow in a tube which is closed on one end. The setup consists of a shock wave traveling towards a wall where it gets reflected. The domain ahead of the incident shock shrinks to a point when the shock hits the wall. After reflection, the shock moves away from the wall and leaves a growing zone behind it, between the wall and the reflected shock.

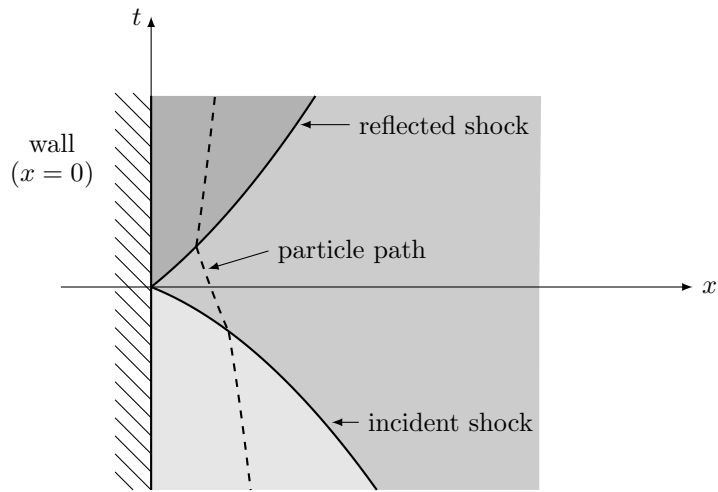
\begin{figure}[h]
  \centering
  \input{physical_picture}  
\caption{Shock reflection.}
\end{figure}
If the state ahead of the incident shock is a zone of quiet, i.e.~fluid at rest with constant density, and behind the incident shock the gas possesses constant velocity and again constant density, the situation can be mathematically represented by piecewise constant solutions of the corresponding differential equations, satisfying the shock conditions across both shocks. See for example the treatment in \cite{courant}. From a point of view of applications, this simple case is of great importance since it represents quite accurately the physical situation in a small neighborhood of a reflection point, giving a good estimate of the pressure increase which can be achieved by reflection. However, from a mathematical point of view the existence of a solution for more general conditions is desirable. 

We start our mathematical treatment of the problem at the time of reflection ($t=0$) and give data for $x\geq 0$ which corresponds to the state behind the incident shock, which at the same time is to be the state ahead of the reflected shock. Thereby we exclude any discussion of the formation, propagation and impingment of the incident shock.
In any case, one can think of the incident shock as being generated by a moving piston (see \cite{courant}) and due to the formation and development results (see \cite{riemann}, \cite{ChristodoulouLisibach}) and an appropriate placement of the wall the setting up to $t=0$ can be rigorously established.

We assume that the given data does not correspond to vanishing velocity at the point of reflection, i.e.~we assume that the data is incompatible with the boundary condition on the wall. Instead we assume that the data at the reflection point, through the jump conditions with vanishing velocity in the state behind, yields a shock speed which is subsonic relative to the state behind and is supersonic relative to the state ahead which is given by the data at the reflection point. These two conditions, while mathematically necessary for a shock to be determined in an evolutionary sense (see \cite{riemann}), correspond also to the only shocks observed naturally. In the following we call these conditions collectively the determinism condition.

Our assumptions consider the data in only one point, the reflection point. Therefore, our assumptions guarantee that the data at the reflection point yields a physical solution to the reflection problem in one point. The existence of such data is therefore justified exactly by the existence of the solution in the case where all the fluid states correspond to constant states as described above.

In addition to the given data we assume the existence of a future development corresponding to a future domain of dependence of $\{x\geq 0\}$ see figure \ref{future_development} on page \pageref{future_development}.

\medskip

The present problem has been treated in a different way in \cite{LiTaTsin1}, \cite{LiTaTsin2}. While in these works the authors treat a more general set of equations, the proof is in stages of increasing difficulty, dealing with a linear system and a fixed boundary first and then extending the results to the nonlinear system and a free boundary. During this process fewer and fewer details of the proof are given. Also, the actual state ahead of the reflected shock is taken into account after the local existence has been established. We believe that our approach and presentation gives additional insight into the problem and opens up the road for future progress in more involved situations.

\section{Equations of Motion}
\subsection{Euler Equations}
We study one dimensional fluid flow without friction. We denote by $\rho$, $w$ and $p$ the density, velocity and pressure respectively. The equations of motion are
\begin{align}
  \label{eq:1}
  \partial_t\rho+\partial_x(\rho w)&=0,\\
  \label{eq:2}
  \partial_t w+w\partial_xw&=-\frac{1}{\rho}\partial_xp.
\end{align}
We assume $p=p(\rho)$ is a given smooth function which satisfies $\frac{dp}{d\rho}(\rho)>0$. We do not take into account entropy. Equations \eqref{eq:1}, \eqref{eq:2} follow from the conservation of mass and momentum once we assume that the quantities describing the fluid are continuously differentiable (see \cite{courant}).

\subsection{Riemann Invariants and Characteristic Equations}
The Riemann invariants are
\begin{align}
  \label{eq:3}
  \alpha\defeq\int^\rho\frac{\eta(\rho')}{\rho'}d\rho'+w,\qquad\beta\defeq\int^\rho\frac{\eta(\rho')}{\rho'}d\rho'-w,
\end{align}
where $\eta\defeq \sqrt{\frac{dp}{d\rho}}$ is the sound speed, see \cite{riemann} and \cite{Earnshaw1998}. As a consequence we have
\begin{align}
  \label{eq:4}
  w=\frac{\alpha-\beta}{2}.
\end{align}
Defining
\begin{alignat}{3}
  \label{eq:5}
  \cout&\defeq w+\eta,&\qquad\qquad\cin&\defeq w-\eta,\\
  \label{eq:6}
  L_{\textrm{out}}&\defeq\partial_t+\cout\partial_x,& L_{\textrm{in}}&\defeq\partial_t+\cin\partial_x,
\end{alignat}
we have
\begin{align}
  \label{eq:7}
  L_{\textrm{out}}\alpha=0=L_{\textrm{in}}\beta,
\end{align}
i.e.~the Riemann invariants $\alpha$, $\beta$ are invariant along the integral curves of $L_\textrm{out}$, $L_{\textrm{in}}$, respectively.

We have
\begin{align}
  \label{eq:8}
  \frac{\partial(\alpha,\beta)}{\partial(\rho,w)}=
  \begin{pmatrix}
    \eta/\rho & 1\\\eta/\rho & -1
  \end{pmatrix}.
\end{align}
Therefore,
\begin{align}
  \label{eq:9}
  \frac{\partial(\rho,w)}{\partial(\alpha,\beta)}=
  \begin{pmatrix}
    \rho/2\eta & \rho/2\eta\\1/2 & -1/2
  \end{pmatrix}.
\end{align}

\subsection{Jump Conditions and the Determinism Condition}
If we assume that there exists a differentiable curve $t\mapsto(t,\xi(t))$ (a so called shock curve) across which the quantities describing the fluid suffer disconinuities but in the closure of both sides the differential equations are satisfied, the conservation of mass and momentum yield the following two conditions (jump conditions) on the discontinuities (see \cite{courant})
\begin{align}
  \label{eq:10}
  \jump{\rho}V&=\jump{\rho w},\\
    \label{eq:11}
  \jump{\rho w}V&=\jump{\rho w^2+p},
\end{align}
where we denote by $V$ the shock speed: $V=d\xi/dt$ and by $\jump{f}$ the difference of the function $f$ across $\xi$, i.e.
\begin{align}
  \label{eq:12}
  \jump{f}=f_+-f_-
\end{align}
where $f_+$ is the quantity evaluated behind the shock and $f_-$ is the quantity evaluated ahead of the shock\footnote{From an evolutionary point of view, looking at the time evolution of a portion of fluid, the state ahead corresponds to the part of the fluid flow line before the intersection whith the shock and the state behind corresponds to the part of the fluid flow line after the intersection with the shock. Hence $\jump{f}$ corresponds to the jump in the quantity $f$ while crossing the shock curve.}:
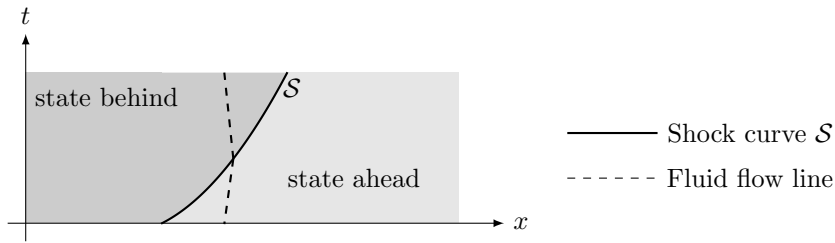
\begin{figure}[h]
  \centering
  \input{shock_notation}
  \caption{The state ahead and behind of the shock.}
\end{figure}

The determinism condition states that the speed of the shock is supersonic relativ to the state ahead of the shock and subsonic relative to the state behind the shock. For a more general description of the determinism condition see the epilogue of \cite{christodoulouformation} or section 2.3 of \cite{ChristodoulouLisibach}.

\section{Setting the Scene}
\subsection{The State Ahead\label{state_ahead}}
We consider data for $\rho$, $w$ given at $t=0$ for $x\geq 0$. This data possesses a future development, i.e.~there exist functions $\rho^\ast(t,x)$, $w^\ast(t,x)$ which solve the equations of motion and coincide with the data for $t=0$. This development of the initial data is bounded in the past by $t=0$, in the future by $t=T$ and to the left by an outgoing characteristic originating at the reflection point $(t,x)=(0,0)$. We are going to denote this characteristic by $\mathcal{B}$. We assume that this data corresponds to a shock reflection point in the following way. We assume that at the reflection point, the data is incompatible with the wall boundary condition, i.e.~we assume $w^\ast(0,0)\neq 0$. We assume that the jump conditions \eqref{eq:10}, \eqref{eq:11} with
\begin{align}
  \label{eq:13}
  \rho_-=\rho^\ast(0,0),\qquad w_-=w^\ast(0,0),\qquad w_+=0
\end{align}
possess the solution
\begin{align}
  \label{eq:14}
  \rho_+\defeq \rho_0,\qquad V\defeq V_0,
\end{align}
such that the determinism condition at $(t,x)=(0,0)$ is satisfied. I.e.
\begin{align}
  \label{eq:15}
  (\cout_0^\ast)_0<V_0<\eta_0,
\end{align}
where
\begin{align}
  \label{eq:16}
  (\cout_0^\ast)_0=w^\ast(0,0)+\eta(\rho^\ast(0,0)),\qquad \eta_0=\eta(\rho_0).
\end{align}
In addition we assume
\begin{align}
  \label{eq:17}
  0<V_0.
\end{align}
The above described development of the data and the solution of the jump conditions at the reflection point, satisfying the determinism condition and \eqref{eq:17}, set up the shock reflection problem.
\begin{figure}[h]
  \centering
  \input{future_development}
  \caption{Future development of initial data.}
  \label{future_development}
\end{figure}
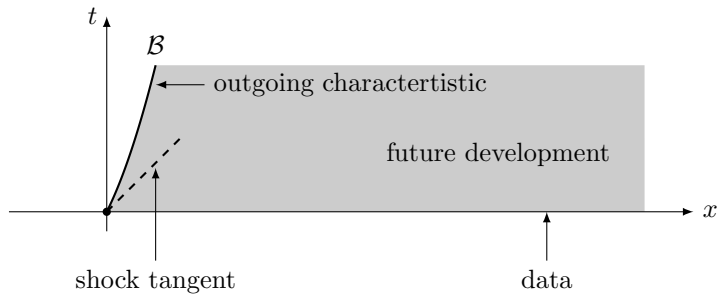

\subsection{The Shock Reflection Problem}
The shock reflection problem is the following:\\
Find a world line $\mathcal{S}$, lying inside the future development, originating at the reflection point $(t,x)=(0,0)$, together with a solution of the equations of motion in a domain of spacetime bounded by $\mathcal{S}$ and $\{x=0\}$, such that along $\{x=0\}$ the wall boundary conditions are satisfied and across $\mathcal{S}$ the new solution displays jumps relative to the solution in the future development of the data, jumps which satisfy the jump conditions. The domain to the right of $\mathcal{S}$, where the solution in the future development of the data holds, is called the state ahead and the domain between the wall $\{x=0\}$ and $\mathcal{S}$, where the new solution holds is called the state behind. $\mathcal{S}$ is to be supersonic relative to the state ahead and subsonic relative to the state behind. The requirement in the last sentence is the determinism condition. We are going to bound the state behind also by an outgoing characteristic.
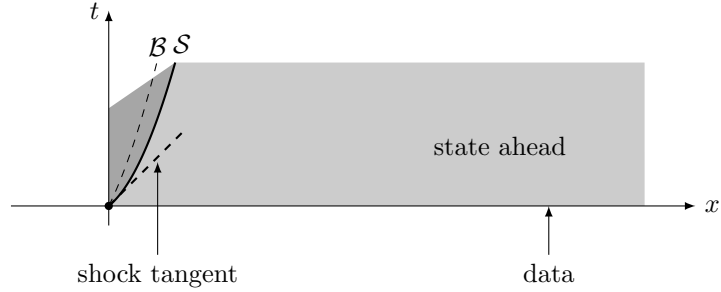
\begin{figure}[h]
  \centering
  \input{reflection_problem}
  \caption{Shock reflection Problem. The state behind the reflected shock in dark shade.}
  \label{reflection_problem}
\end{figure}

\section{Characteristic Coordinates}
\subsection{Choice of Coordinates}
Introducing coordinates $u$, $v$\label{u-v-coordinates}, such that $u$ is constant along integral curves of $L_{\textrm{out}}$ and $v$ is constant along integral curves of $L_{\textrm{in}}$ (see \eqref{eq:6}), \eqref{eq:7} becomes
\begin{align}
  \label{eq:18}
  \frac{\partial\alpha}{\partial v}=0=\frac{\partial\beta}{\partial u}.
\end{align}
In $u$-$v$-coordinates, $t$ and $x$ satisfy the characteristic equations
\begin{align}
  \label{eq:19}
  \frac{\partial x}{\partial v}=\cout\frac{\partial t}{\partial v},\qquad\frac{\partial x}{\partial u}=\cin\frac{\partial t}{\partial u}.
\end{align}

We choose the $u$-$v$-coordinates such that the following conditions hold:
\begin{enumerate}
\item The origin $(u,v)=(0,0)$ corresponds to the reflection point $(t,x)=(0,0)$.
\item The wall $\{x=0\}$ corresponds to $u=v$.
\item The shock curve corresponds to $u=av$, for a constant $a$ with $0<a<1$ (see \eqref{eq:33}, \eqref{eq:35} below).
\item We have
  \begin{align}
    \label{eq:20}
    \pp{x}{v}(0,0)=1.
  \end{align}
\end{enumerate}

This choice of coordinates is justified as follows: The condition $u=v$ at the wall $\{x=0\}$ can be imposed once a $v=\textrm{const.}$~curve intersects the wall exactly once, which is the case because the wall is a subsonic curve. Thus $u$ is determined once $v$ is determined. The shock curve being subsonic relative to the state behind, it is given in $u$-$v$-coordinates by an equation of the form $u=f(v)$, where $f$ is an increasing function. We must have $f(0)=0$ since the reflection point, which is the origin, is on the wall as well as on the shock curve. Moreover, $f(v)<v$ for $v>0$ because the shock curve is to the right of the wall. Now, we can set $u=\phi(\tilde{u})$, $v=\phi(\tilde{v})$, where $\phi$ is any increasing function such that $\phi(0)=0$, without changing the equation of the wall $\tilde{u}=\tilde{v}$ or the fact that the reflection point corresponds to the origin. The equation of the shock curve is then transformed to $\tilde{u}=\tilde{f}(\tilde{v})$, where
\begin{align}
  \label{eq:21}
  \tilde{f}=\phi^{-1}\circ f\circ \phi.
\end{align}
It follows that $\tilde{f}'(0)=f'(0)$. Let $a=f'(0)$. We have $0<a<1$. The problem is then to choose an appropriate $\phi$ such that $\tilde{f}(\tilde{v})=a\tilde{v}$. It can easily be shown by an iteration method starting with the $0$'th iterate $\phi_0$ being the identity map $\phi_0(x)=x$ that the equation $\phi\circ \tilde{f}=f\circ \phi$ with $\tilde{f}(x)=ax$ has a solution $\phi$ defined on $[0,\varepsilon^\ast]$ for suitably small $\varepsilon^\ast>0$. We can then extend this local solution to a global one using a continuity argument.

Let $\varepsilon>0$. In the following we consider the domain which is bounded by the wall, the shock curve and the outgoing characteristic $u=a\varepsilon$:
\begin{align}
  \label{eq:22}
  T_\varepsilon=\Big\{(u,v)\in\mathbb{R}^2:0\leq u\leq v\leq \frac{u}{a}\leq\varepsilon\Big\}.
\end{align}
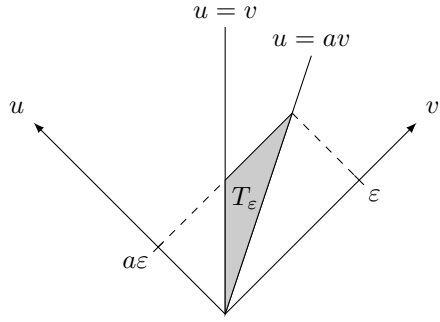
\begin{figure}[h]
  \centering
  \input{u-v-coordinates}
  \caption{The domain $T_\varepsilon$.}
\end{figure}

\subsection{Boundary Conditions along the Wall}
Along the wall $u=v$ we have
\begin{align}
  \label{eq:23}
  w(u,u)=0.
\end{align}
Using \eqref{eq:4}, this implies
\begin{align}
  \label{eq:24}
  \alpha(u,u)=\beta(u,u).
\end{align}

The function $x(u,v)$ satisfies the boundary condition
\begin{align}
  \label{eq:25}
  x(u,u)&=0.
\end{align}

\subsection{Boundary Conditions along the Shock Curve}
The functions $t(u,v)$, $x(u,v)$ along the shock curve $u=av$ are
\begin{align}
  \label{eq:26}
  t_+(v)\defeq t(av,v),\qquad x_+(v)\defeq x(av,v).
\end{align}
The shock speed $V$ satisfies
\begin{align}
  \label{eq:27}
  Vdt_+=dx_+.
\end{align}
Therefore,
\begin{align}
  \label{eq:28}
  V(v)\left(a\pp{t}{u}(av,v)+\pp{t}{v}(av,v)\right)=a\pp{x}{u}(av,v)+\pp{x}{v}(av,v).
\end{align}
Using \eqref{eq:19} we rewrite this as
\begin{align}
  \label{eq:29}
  \pp{x}{v}(av,v)=a\pp{x}{u}(av,v)\frac{\cout_+(v)}{\cin_+(v)}\,\frac{V(v)-\cin_+(v)}{\cout_+(v)-V(v)}.
\end{align}
At the reflection point $(u,v)=(0,0)$ we have
\begin{align}
  \label{eq:30}
  \pp{x}{v}(0,0)=1=-\pp{x}{u}(0,0)
\end{align}
(see \eqref{eq:20}, \eqref{eq:25}) and
\begin{align}
  \label{eq:31}
  \cin(0,0)&=w(0,0)-\eta(0,0)=-\eta_0,\\
    \label{eq:32}
  \cout(0,0)&=w(0,0)+\eta(0,0)=\eta_0,
\end{align}
where $\eta_0=\eta(\rho_0)$ is given by the solution of the jump conditions at the reflection point (see \eqref{eq:14}). Hence,
\begin{align}
  \label{eq:33}
    a=\frac{\eta_0-V_0}{\eta_0+V_0},
\end{align}
which fixes the value of the constant $a$ introduced in the assumptions on the coordinate system. By the assumptions \eqref{eq:15}, \eqref{eq:17} we have
\begin{align}
  \label{eq:34}
0<V_0<\eta_0,
\end{align}
which implies
\begin{align}
  \label{eq:35}
  0<a<1,
\end{align}
compatible with our choice of coordinates. Defining
\begin{align}
  \label{eq:36}
    \Gamma(v)\defeq a\frac{\cout_+(v)}{\cin_+(v)}\,\frac{V(v)-\cin_+(v)}{\cout_+(v)-V(v)},
\end{align}
the boundary condition \eqref{eq:29} becomes
\begin{align}
  \label{eq:37}
  \pp{x}{v}(av,v)=\pp{x}{u}(av,v)\Gamma(v).
\end{align}
We note that
\begin{align}
  \label{eq:38}
  \Gamma(0)=-1.
\end{align}

The functions $\alpha(u,v)$, $\beta(u,v)$ are given along the shock by
\begin{align}
  \label{eq:39}
  \alpha_+(v)\defeq\alpha(av,v),\qquad \beta_+(v)\defeq\beta(av,v).
\end{align}
Together with \eqref{eq:18} and \eqref{eq:24}, we obtain
\begin{align}
  \label{eq:40}
  \alpha_+(v)=\alpha(av,v)=\alpha(av,av)=\beta(av,av)=\beta(a^2v,av)=\beta_+(av).
\end{align}
Therefore,
\begin{align}
  \label{eq:41}
\alpha_+'(v)&=a\beta_+'(av).
\end{align}

\subsection{Jump Conditions}
The jump conditions \eqref{eq:10}, \eqref{eq:11} are equivalent to
\begin{align}
  \label{eq:42}
  V&=\frac{\jump{\rho w}}{\jump{\rho}},\\
    \label{eq:43}
  0&=\jump{\rho w}^2-\jump{\rho w^2+p}\jump{\rho}\defeq I(\rho_+,\rho_-,w_+,w_-).
\end{align}
$\rho$ and $w$ are given smooth functions of $\alpha$, $\beta$. Defining
\begin{align}
  \label{eq:44}
  J(\alpha_+,\beta_+,\alpha_-,\beta_-)\defeq I\Big(\rho(\alpha_+,\beta_+),\rho(\alpha_-,\beta_-),w(\alpha_+,\beta_+),w(\alpha_-,\beta_-)\Big),
\end{align}
the above jump condition \eqref{eq:43} in terms of the Riemann invariants is
\begin{align}
  \label{eq:45}
  J(\alpha_+,\beta_+,\alpha_-,\beta_-)=0.
\end{align}
We recall that the solution of the jump conditions \eqref{eq:42}, \eqref{eq:43} at the reflection point is $\rho_0$, $V_0$ (see \eqref{eq:14}). The value $\rho_0$ determines the value of $\beta_0=\alpha_0$. These values, substituted for $\alpha_+$, $\beta_+$, together with the values of $\alpha$, $\beta$ at the reflection point in the state ahead (hence given by the data) satisfy the jump condition \eqref{eq:45}, i.e.
\begin{align}
  \label{eq:46}
  J(\alpha_0,\beta_0,\alpha_{-0},\beta_{-0})=0.
\end{align}
Here
\begin{align}
  \label{eq:47}
  \alpha_{-0}=\alpha^\ast(0,0),\qquad \beta_{-0}=\beta^\ast(0,0)
\end{align}
are given by the solution in the state ahead at the reflection point. We have
\begin{align}
  \label{eq:48}
  \pp{J}{\beta_+}(\alpha_+,\beta_+,\alpha_-,\beta_-)&=\pp{I}{\rho_+}\Big(\rho(\alpha_+,\beta_+),\rho(\alpha_-,\beta_-),w(\alpha_+,\beta_+),w(\alpha_-,\beta_-)\Big)\pp{\rho}{\beta}(\alpha_+,\beta_+)\nonumber\\
  &\quad +\pp{I}{w_+}\Big(\rho(\alpha_+,\beta_+),\rho(\alpha_-,\beta_-),w(\alpha_+,\beta_+),w(\alpha_-,\beta_-)\Big)\pp{w}{\beta}(\alpha_+,\beta_+).
\end{align}
We have
\begin{align}
\label{eq:49}
  \pp{I}{\rho_+}(\rho_+,\rho_-,w_+,w_-)&=2\jump{\rho w}w_+-(w_+^2+\eta_+^2)\jump{\rho}-\jump{\rho w^2+p},\\
\label{eq:50}
  \pp{I}{w_+}(\rho_+,\rho_-,w_+,w_-)&=2\jump{\rho w}\rho_+-2\rho_+w_+\jump{\rho}.
\end{align}
Using these together with \eqref{eq:9} and the jump conditions \eqref{eq:42}, \eqref{eq:43} in \eqref{eq:48} we arrive at
\begin{align}
  \label{eq:51}
  \pp{J}{\beta_+}(\alpha_+,\beta_+,\alpha_-,\beta_-)&=-\frac{\jump{\rho}\rho_+}{2\eta_+}(V-\cin_+)^2.
\end{align}
In this equation all the quantities on the right hand side are functions of $\alpha_+$, $\beta_+$, $\alpha_-$, $\beta_-$, either through given functions of $\rho$ and $w$ or, as in the case for $V$, through the jump condition \eqref{eq:42}. We have
\begin{align}
  \label{eq:52}
  V_0-\cin_{+0}>-\cin_{+0}=-(w_0-\eta_0)=\eta_0>0,
\end{align}
where for the first inequality we used \eqref{eq:17} (recall also that $w_0=0$). Hence
\begin{align}
  \label{eq:53}
  \pp{J}{\beta_+}(\alpha_0,\beta_0,\alpha_{-0},\beta_{-0})\neq 0.
\end{align}
Using the implicit function theorem, we conclude from  \eqref{eq:46}, \eqref{eq:53} that there exists a smooth function $H(\alpha_+,\alpha_-,\beta_-)$, such that
\begin{align}
  \label{eq:54}
  \beta_0=H(\alpha_0,\alpha_{-0},\beta_{-0})
\end{align}
and
\begin{align}
  \label{eq:55}
  J\Big(\alpha_+,H(\alpha_+,\alpha_{-},\beta_{-}),\alpha_-,\beta_-\Big)=0
\end{align}
for $(\alpha_+,\alpha_{-},\beta_{-})$ sufficiently close to $(\alpha_0,\alpha_{-0},\beta_{-0})$. Along the shock $u=av$ we have
\begin{align}
  \label{eq:56}
  \beta_+(v)=H\Big(\alpha_+(v),\alpha_-(v),\beta_-(v)\Big).
\end{align}
Taking the derivative we obtain
\begin{align}
  \label{eq:57}
  \beta_+'(v)&=F\Big(\alpha_+(v),\alpha_-(v),\beta_-(v)\Big)\alpha_+'(v)\nonumber\\
             &\quad+M_1\Big(\alpha_+(v),\alpha_-(v),\beta_-(v)\Big)\beta_-'(v)\nonumber\\
             &\quad+M_2\Big(\alpha_+(v),\alpha_-(v),\beta_-(v)\Big)\alpha_-'(v),
\end{align}
where
\begin{align}
  \label{eq:58}
  F(\alpha_+,\alpha_-,\beta_-)&\defeq\pp{H}{\alpha_+}(\alpha_+,\alpha_-,\beta_-)\nonumber\\
                                       &=-\frac{\displaystyle{\pp{J}{\alpha_+}(\alpha_+,H(\alpha_+,\alpha_-,\beta_-),\alpha_-,\beta_-)}}{\displaystyle{\pp{J}{\beta_+}(\alpha_+,H(\alpha_+,\alpha_-,\beta_-),\alpha_-,\beta_-)}}
\end{align}
and similar expressions hold for $M_1$, $M_2$. The denominator in \eqref{eq:58} is given by \eqref{eq:51}. For the numerator we find analogously
\begin{align}
  \label{eq:59}
  \pp{J}{\alpha_+}(\alpha_+,\beta_+,\alpha_-,\beta_-)=-\frac{\jump{\rho}\rho_+}{2\eta_+}(\cout_+-V)^2.
\end{align}
Hence, we obtain
\begin{align}
  \label{eq:60}
  F(\alpha_+,\alpha_{-},\beta_{-})=-\left(\frac{\cout_+-V}{V-\cin_+}\right)^2.
\end{align}
We note that by \eqref{eq:31}, \eqref{eq:32}, \eqref{eq:33} we have
\begin{align}
  \label{eq:61}
  F_0\defeq F(\alpha_0,\alpha_{-0},\beta_{-0})=-a^2.
\end{align}
Taking another derivative of \eqref{eq:57} we obtain
\begin{align}
  \label{eq:62}
  \beta_+''(v)&=F\Big(\alpha_+(v),\alpha_-(v),\beta_-(v)\Big)\alpha_+''(v)\nonumber\\
  &\quad+M_1\Big(\alpha_+(v),\alpha_-(v),\beta_-(v)\Big)\beta_-''(v)\nonumber\\
  &\quad+M_2\Big(\alpha_+(v),\alpha_-(v),\beta_-(v)\Big)\alpha_-''(v)\nonumber\\
  &\quad+G\Big(\alpha_+(v),\alpha_-(v),\beta_-(v),\alpha_+'(v),\alpha_-'(v),\beta_-'(v)\Big),
\end{align}
where
\begin{align}
  \label{eq:63}
  G\Big(\alpha_+,\alpha_-,\beta_-,\alpha_+',\alpha_-',\beta_-'\Big)&\defeq\left(\pp{F}{\alpha_+}\alpha_+'+\pp{F}{\alpha_-}\alpha_-'+\pp{F}{\beta_-}\beta_-'\right)\alpha_+'\nonumber\\
                                                                                     &\quad+\left(\pp{M_1}{\alpha_+}\alpha_+'+\pp{M_1}{\alpha_-}\alpha_-'+\pp{M_1}{\beta_-}\beta_-'\right)\beta_-'\nonumber\\
                                                                                     &\quad +\left(\pp{M_2}{\alpha_+}\alpha_+'+\pp{M_2}{\alpha_-}\alpha_-'+\pp{M_2}{\beta_-}\beta_-'\right)\alpha_-'.
\end{align}
We note that the functions $F$, $M_1$, $M_2$ and $G$ are smooth functions of their arguments.

Setting $v=0$ in \eqref{eq:57} we obtain (recall that $\alpha_+(0)=\beta_0$)
\begin{align}
  \label{eq:64}
  \beta_0'\defeq\beta_+'(0)&=F\Big(\beta_0,\alpha_{-0},\beta_{-0}\Big)\alpha_+'(0)\nonumber\\
             &\quad+M_1\Big(\beta_0,\alpha_{-0},\beta_{-0}\Big)\beta_-'(0)\nonumber\\
             &\quad+M_2\Big(\beta_0,\alpha_{-0},\beta_{-0}\Big)\alpha_-'(0).
\end{align}
We have
\begin{align}
  \label{eq:65}
  \alpha_-(v)=\alpha^\ast(t_+(v),x_+(v)),\qquad \beta_-(v)=\beta^\ast(t_+(v),x_+(v)).
\end{align}
Taking the derivative yields
\begin{align}
  \label{eq:66}
  \alpha_-'(0)&=\left(\pp{\alpha^\ast}{t}\right)_0t_+'(0)+\left(\pp{\alpha^\ast}{x}\right)_0x_+'(0),\\
  \label{eq:67}
  \beta_-'(0)&=\left(\pp{\beta^\ast}{t}\right)_0t_+'(0)+\left(\pp{\beta^\ast}{x}\right)_0x_+'(0).
  \end{align}
The partial derivatives of $\alpha^\ast(t,x)$, $\beta^\ast(t,x)$ are given by the solution in the state ahead. With
\begin{align}
  \label{eq:68}
  t_+'(v)=a\pp{t}{u}(av,v)+\pp{t}{v}(av,v),\qquad x_+'(v)=a\pp{x}{u}(av,v)+\pp{x}{v}(av,v)
\end{align}
and using the characteristic equations \eqref{eq:19} together with
\begin{align}
  \label{eq:69}
  \cout(0,0)=\eta_0=-\cin(0,0)
\end{align}
and \eqref{eq:20} we obtain
\begin{align}
  \label{eq:70}
  \alpha_-'(0)&=\left(\pp{\alpha^\ast}{t}\right)_0\frac{1+a}{\eta_0}+\left(\pp{\alpha^\ast}{x}\right)_0(1-a),\\
  \label{eq:71}
  \beta_-'(0)&=\left(\pp{\beta^\ast}{t}\right)_0\frac{1+a}{\eta_0}+\left(\pp{\beta^\ast}{x}\right)_0(1-a). 
\end{align}
Using  \eqref{eq:61} and $\alpha_+'(0)=a\beta_+'(0)=a\beta_0'$ (see \eqref{eq:41}) for the first term in \eqref{eq:64}  we obtain
\begin{align}
  \label{eq:72}
  \beta_0'=\frac{1}{1+a^3}\left(M_1\Big(\beta_0,\alpha_{-0},\beta_{-0}\Big)\beta_-'(0)+M_2\Big(\beta_0,\alpha_{-0},\beta_{-0}\Big)\alpha_-'(0)\right)
\end{align}
with $\alpha_-'(0)$, $\beta_-'(0)$ given by \eqref{eq:70}, \eqref{eq:71} respectively.

In the following we use \eqref{eq:54} and \eqref{eq:72} as definitions of the constants $\beta_0$, $\beta_0'$. We note that these constants are given in terms of the solution in the state ahead at the reflection point and the solution $\rho_0$ as given by the jump conditions at the reflection point, together with the condition $w_+=0$ (see \eqref{eq:13}, \eqref{eq:14}).

\subsection{Equations for $\pp{x}{u}$, $\pp{x}{v}$}
In the following we consider a point in $T_\varepsilon$ with coordinates $(u,v)$. We have
\begin{align}
  \label{eq:73}
  \pp{x}{u}(u,v)=\pp{x}{u}(u,u)+\int_u^v\pppp{x}{u}{v}(u,v')dv'.
\end{align}
This corresponds to an integration along an outgoing characteristic, starting on the wall and ending at $(u,v)$. We rewrite the first term on the right as
\begin{align}
  \label{eq:74}
  \pp{x}{u}(u,u)&=-\pp{x}{v}(u,u)\nonumber\\
                &=-\pp{x}{v}(au,u)-\int_{au}^u\pppp{x}{u}{v}(u',u)du'\nonumber\\
                &=-\Gamma(u)\pp{x}{u}(au,u)-\int_{au}^u\pppp{x}{u}{v}(u',u)du'.
\end{align}
where for the first equality we used the boundary condition on the wall, see \eqref{eq:25}, for the second equality we integrate along an incoming characteristic starting at the shock and ending on the wall and for the last equality we use the boundary condition at the shock, see \eqref{eq:37}. Substituting \eqref{eq:74} into \eqref{eq:73} we obtain
\begin{align}
  \label{eq:75}
  \pp{x}{u}(u,v)=-\Gamma(u)\pp{x}{u}(au,u)-\int_{au}^u\pppp{x}{u}{v}(u',u)du'+\int_u^v\pppp{x}{u}{v}(u,v')dv'.
\end{align}
The paths of integration, as given by the two integrals in this equation, are shown in figure \ref{integration_paths} on the left.

We have
\begin{align}
  \label{eq:76}
  \pp{x}{v}(u,v)=\pp{x}{v}(av,v)+\int_{av}^u\pppp{x}{u}{v}(u',v)du'.
\end{align}
This corresponds to integrating along an incoming characteristic from the shock to the point $(u,v)$. For the first term we use \eqref{eq:37}. We obtain
\begin{align}
  \label{eq:77}
  \pp{x}{v}(u,v)=\Gamma(v)\pp{x}{u}(av,v)+\int_{av}^u\pppp{x}{u}{v}(u',v)du'.
\end{align}
The path of integration, as given by the integral in this equation, is shown in figure \ref{integration_paths} on the right.

\tikzset{->-/.style={decoration={
  markings,
  mark=at position .5 with {\arrow{>}}},postaction={decorate}}}

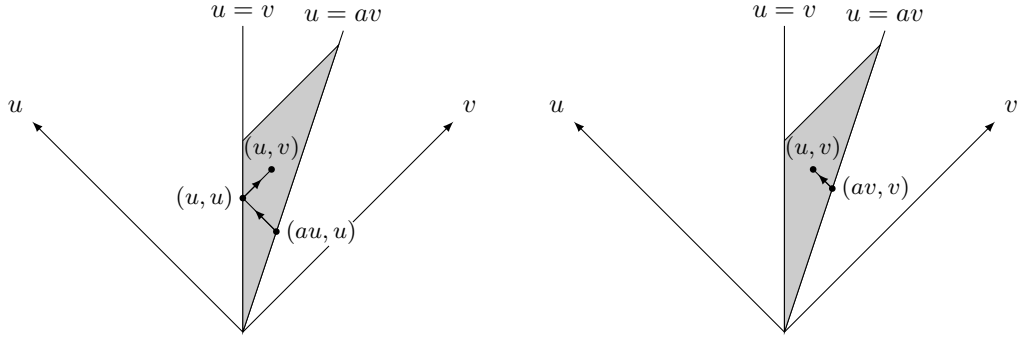
\begin{figure}[h!]
  \centering
  \input{delxdelu}
  \caption{Integration paths for $\pp{x}{u}(u,v)$ on the left, as in equation \eqref{eq:75}. Integration path for $\pp{x}{v}(u,v)$ on the right, as in equation \eqref{eq:77}.
  \label{integration_paths}}
\end{figure}

Taking the derivative of the first of \eqref{eq:19} with respect to $u$ and of the second of \eqref{eq:19} with respect to $v$ and subtracting the resulting expressions, we obtain
\begin{align}
  \label{eq:78}
  \pppp{x}{u}{v}=\frac{1}{\cout-\cin}\left(\frac{\cout}{\cin}\pp{\cin}{v}\pp{x}{u}-\frac{\cin}{\cout}\pp{\cout}{u}\pp{x}{v}\right).
\end{align}
Defining
\begin{align}
  \label{eq:79}
  \mu\defeq \frac{1}{\cout-\cin}\,\frac{\cout}{\cin}\pp{\cin}{v},\qquad \nu\defeq-\frac{1}{\cout-\cin}\,\frac{\cin}{\cout}\pp{\cout}{u}.
\end{align}
this is
\begin{align}
  \label{eq:80}
  \pppp{x}{u}{v}(u,v)=\mu(u,v)\pp{x}{u}(u,v)+\nu(u,v)\pp{x}{v}(u,v).
\end{align}

\subsection{Formulation of the Problem\label{shock_reflection_problem}}
The shock reflection problem is the following: Find a solution of the equations
\begin{align}
  \label{eq:81}
  \pp{x}{u}&=\cin\pp{t}{u},\\
  \label{eq:82}
  \pp{x}{v}&=\cout\pp{t}{v},\\
  \label{eq:83}
  \pp{\alpha}{v}&=0,\\
  \label{eq:84}
  \pp{\beta}{u}&=0
\end{align}
in $T_\varepsilon$, such that along $u=v$ the wall boundary conditions
\begin{align}
  \label{eq:85}
  \alpha(u,u)=\beta(u,u),\qquad x(u,u)=0,
\end{align}
and along $u=av$ the shock boundary condition
\begin{align}
  \label{eq:86}
  \frac{dx_+}{dv}(v)=V(v)\frac{dt_+}{dv}(v)
\end{align}
is satisfied. Here
\begin{align}
  \label{eq:87}
  x_+(v)=x(av,v),\qquad t_+(v)=t(av,v)
\end{align}
and $V$ is given by
\begin{align}
  \label{eq:88}
  V=\frac{\jump{\rho w}}{\jump{\rho}},
\end{align}
where
\begin{align}
  \label{eq:89}
  \rho_\pm(v)=\rho(\alpha_\pm(v),\beta_\pm(v)),\qquad w_\pm(v)=w(\alpha_\pm,\beta_\pm),
\end{align}
where $\rho(\alpha,\beta)$, $w(\alpha,\beta)$ are given smooth functions of their arguments and
\begin{alignat}{3}
  \label{eq:90}
  \alpha_+(v)&=\alpha(av,v),&\qquad\beta_+(v)&=\beta(av,v),\\
  \label{eq:91}
  \alpha_-(v)&=\alpha^\ast(t_+(v),x_+(v)),&\beta_-(v)&=\beta^\ast(t_+(v),x_+(v)).
\end{alignat}
Furthermore, along $u=av$ the jump condition
\begin{align}
  \label{eq:92}
  J(\alpha_+,\beta_+,\alpha_-,\beta_-)=0
\end{align}
is satisfied and at the point of reflection $(u,v)=(0,0)$ we have
\begin{align}
  \label{eq:93}
  \beta(0,0)=\beta_0,\qquad \pp{\beta}{v}(0,0)=\beta_0',\qquad V(0)=V_0.
\end{align}
In addition, the determinism condition has to be satisfied:
\begin{align}
  \label{eq:94}
  \cout_-(v)<V(v)<\cout_+(v),
\end{align}
where
\begin{align}
  \label{eq:95}
  \cout_-(v)=\cout(\alpha_-(v),\beta_-(v)),\qquad\cout_+(v)=\cout(\alpha_+(v),\beta_+(v)).
\end{align}

\section{Solution of the Reflection Problem}
We have the following theorem:
\begin{theorem}[Existence]\label{existence_theorem}
  For $\varepsilon$ sufficiently small, the shock reflection problem as stated in the previous subsection possesses a solution $\alpha$, $\beta$, $t$, $x$ in $C^2(T_\varepsilon)$. Also, for $\varepsilon$ sufficiently small, the Jacobian $\frac{\partial(t,x)}{\partial(u,v)}$ does not vanish in $T_\varepsilon$, which implies that $\alpha$, $\beta$ are $C^2$ functions of $(t,x)$ on the image of $T_\varepsilon$ by the map $(u,v)\mapsto (t(u,v),x(u,v))$.
\end{theorem}
The proof of this theorem is based on an iteration scheme employed in the following subsections.

\subsection{Setup of the Iteration Scheme}
We solve the reflection problem using an iteration. We base our iteration scheme on the functions $x(u,v)$ and $\beta_+(v)$. The iteration scheme is as follows:
\begin{enumerate}
\item We initiate the sequence by\footnote{The index $0$ on the left of these equations denotes the $0$'th iterate in our sequence of functions. This is the only place where the index $0$ is used in this way, everywhere else it is used to denote evaluation of quantities at the point of reflection (which corresponds to the origin of our coordinates), hence no confusion should arise.}
  \begin{align}
    \label{eq:96}
    x_0(u,v)=v-u,\qquad \beta_{+0}(v)=\beta_0+\beta_0'v.
  \end{align}

\item We start with approximate solutions $x_m\in C^2(T_\varepsilon)$, $\beta_{+m}\in C^2[0,\varepsilon]$.
\item We set
  \begin{align}
    \label{eq:97}
    \beta_m(u,v)=\beta_{+m}(v),\qquad\alpha_m(u,v)=\beta_{+m}(u).    
  \end{align}
\item We set
  \begin{align}
  \label{eq:98}
    t_m(u,v)=\int_0^u\left(\phi_m+\psi_m\right)(u',u')du'+\int_u^v\psi_m(u,v')dv',
  \end{align}
  where we set
\begin{align}
  \label{eq:99}
  \phi_m(u,v)&=\frac{1}{\cin_m(u,v)}\pp{x_m}{u}(u,v),\\
    \label{eq:100}
  \psi_m(u,v)&=\frac{1}{\cout_m(u,v)}\pp{x_m}{v}(u,v).
\end{align}
Here we use the notation
\begin{align}
  \label{eq:101}
  \cin_m(u,v)=\cin(\alpha_m(u,v),\beta_m(u,v)),\qquad\cout_m(u,v)=\cout(\alpha_m(u,v),\beta_m(u,v)).
\end{align}

\item We then compute the quantities appearing in the jumps evaluated in the state ahead:
\begin{align}
  \label{eq:102}
  \alpha_{-m}(v)&=\alpha^\ast(t_{+m}(v),x_{+m}(v)),\\
    \label{eq:103}
  \beta_{-m}(v)&=\beta^\ast(t_{+m}(v),x_{+m}(v)),
\end{align}
where
\begin{align}
  \label{eq:104}
  t_{+m}(v)=t_m(av,v),\qquad x_{+m}(v)=x_m(av,v)
\end{align}
and $\alpha^\ast(t,x)$, $\beta^\ast(t,x)$ are given by the solution in the state ahead.
\item The quantities in the sate behind are given by
  \begin{align}
    \label{eq:105}
    \beta_{+m}(v)=\beta_m(av,v),\qquad \alpha_{+m}(v)=\alpha_m(av,v)=\beta_{+m}(av).
  \end{align}
Hence all quantities appearing in the jumps are given and we compute $V_m$ using
\begin{align}
  \label{eq:106}
  V_m=\frac{\jump{\rho_m w_m}}{\jump{\rho_m}},
\end{align}
where $\jump{f_m}=f_{+m}-f_{-m}$ and
\begin{align}
  \label{eq:107}
  \rho_{\pm m}(v)=\rho(\alpha_{\pm m}(v),\beta_{\pm m}(v)),\qquad w_{\pm m}(v)=w(\alpha_{\pm m}(v),\beta_{\pm m}(v))
\end{align}

From this we compute
  \begin{align}
    \label{eq:108}
    \Gamma_m= a\frac{{\cout}_{+m}}{{\cin}_{+m}}\,\frac{V_m-{\cin}_{+m}}{{\cout}_{+m}-V_m},
  \end{align}
where
\begin{align}
  \label{eq:109}
  \cin_{+m}(v)=\cin(\alpha_{+m}(v),\beta_{+m}(v)),\qquad\cout_{+m}(v)=\cout(\alpha_{+m}(v),\beta_{+m}(v))
\end{align}

\item We set
  \begin{align}
    \label{eq:110}
    x_{m+1}(u,v)=v-u+\int_u^v\Phi_{m}(u,v')dv',
  \end{align}
where
\begin{align}
  \label{eq:111}
  \Phi_{m}(u,v)=\int_0^v\Lambda_m(u')du'+\int_{av}^uM_m(u',v)du',
\end{align}
    where we set
    \begin{align}
    \label{eq:112}
      M_m(u,v)&=\mu_{m}(u,v)\pp{x_{m}}{u}(u,v)+\nu_{m}(u,v)\pp{x_{m}}{v}(u,v),\\
      \label{eq:113}
      \Lambda_m(u)&=\Gamma_m(u)a\ppp{x_m}{u}(au,u)+\Gamma'(u)\pp{x_m}{u}(au,u)+\Gamma_m(u)M_m(au,u).
    \end{align}
    where the coefficients $\mu_m$, $\nu_m$ are given in terms of $\alpha_m$, $\beta_m$ and their derivatives.
\item We compute $\beta_{+,m+1}(v)$ as
  \begin{align}
    \label{eq:114}
\beta_{+,m+1}(v)=H\Big(\alpha_{+m}(v),\alpha_{-m}(v),\beta_{-m}(v)\Big).
  \end{align}
\end{enumerate}

\subsection{Remarks on the Iteration Scheme}
\begin{enumerate}
\item The way we set things up, to each pair of functions $x_m$, $\beta_{+m}$ there corresponds a unique pair of functions $t_m$, $\alpha_{+m}$. It therefore suffices to show that the iteration mapping maps the respective spaces to itself (by induction) and the convergence only for the sequence $((x_m,\beta_{+m});m=0,1,2,\ldots)$.
\item From \eqref{eq:97} we see that for every member of the sequence the equations \eqref{eq:83}, \eqref{eq:84} as well as the first wall boundary condition in \eqref{eq:85} are satisfied. I.e.
\begin{align}
  \label{eq:115}
  \pp{\alpha_m}{v}=0,\qquad \pp{\beta_m}{u}=0,\qquad \alpha_m(u,u)=\beta_m(u,u).
\end{align}
Also we see that by \eqref{eq:110} the second wall boundary condition in \eqref{eq:85} is satisfied. I.e.
\begin{align}
  \label{eq:116}
   x_m(u,u)=0.
\end{align}
\item It will be shown in section \ref{inductive_step} that taking the partial derivative of \eqref{eq:110} with respect to $u$ yields an equation of the form
\begin{align}
  \label{eq:117}
  \ppp{x_{m+1}}{u}(u,v)=-\Gamma_m(u)a\ppp{x_m}{u}(au,u)+\ldots
\end{align}
The terms not written out are either of mixed derivative type, of lower order, or mixed type third derivatives but appearing in an integral. Equation \eqref{eq:117} corresponds to \eqref{eq:75} once a partial derivative with respect to $u$ is taken and the $m+1$'th iterate is placed in the left hand side and the $m$'th iterate is placed in the right hand side. Since $0<a<1$ and $\Gamma_m(0)=-1$, the overall prefactor in front of $\ppp{x_m}{u}(au,u)$ can be made strictly smaller than one, which is crucial in establishing that the sequence of iterates is suitably confinded and that it is convergent.
\item In step 5 of our iteration scheme we have to make sure that the shock curve corresponding to the $m$'th iterate: $v\mapsto (t_{+m}(v),x_{+m}(v))$ lies in the future development of the data. Otherwise the evaluation of the quantities $\alpha^\ast(t_{+m}(v),x_{+m}(v))$, $\beta^\ast(t_{+m}(v),x_{+m}(v))$ would be meaningless.
\end{enumerate}

\subsection{Inductive Step\label{inductive_step}}
We choose closed balls in function spaces as follows:
\begin{align}
  \label{eq:118}
  B_{B}&=\Big\{f\in C^2[0,\varepsilon]:f(0)=\beta_0,f'(0)=\beta_0',\|f\|_1\leq B\Big\},\\
    \label{eq:119}
  B_{N_0}&=\Big\{f\in C^2(T_\varepsilon):f(0,0)=0,-\pp{f}{u}(0,0)=1=\pp{f}{v}(0,0),\|f\|_2\leq N_0\Big\},
\end{align}
where we use
\begin{align}
  \label{eq:120}
  \|f\|_1&\defeq\sup_{[0,\varepsilon]}|f''|,\\
  \label{eq:121}
  \|f\|_2&\defeq\max\left\{\sup_{T_\varepsilon}\left|\ppp{f}{u}\right|,\sup_{T_\varepsilon}\left|\pppp{f}{u}{v}\right|,\sup_{T_\varepsilon}\left|\ppp{f}{v}\right|\right\}.
\end{align}
For the constants $\beta_0$, $\beta_0'$ see \eqref{eq:54}, \eqref{eq:72} respectively.
\begin{proposition}
  Choosing the constants $B$, $N_0$ appropriately, the sequence
  \begin{align}
    \label{eq:122}
    ((\beta_{+m},x_m);m=0,1,2,\ldots)   
  \end{align}
is contained in $B_B\times B_{N_0}$, provided we choose $\varepsilon$ sufficientely small.
\end{proposition}
\begin{proof}
  Since we initiate the sequence by (see \eqref{eq:96})
\begin{align}
  \label{eq:123}
  x_0(u,v)=v-u,\qquad \beta_{+0}(v)=\beta_0+\beta_0'v,
\end{align}
we see that
\begin{align}
  \label{eq:124}
  (\beta_{+0},x_0)\in B_B\times B_{N_0}.
\end{align}

Let now
\begin{align}
  \label{eq:125}
  (\beta_{+m},x_m)\in B_B\times B_{N_0}.
\end{align}
We have to show that from this it follows that
\begin{align}
  \label{eq:126}
  (\beta_{+,m+1},x_{m+1})\in B_B\times B_{N_0}.
\end{align}

We first derive estimates for $\alpha_m(u,v)$ and $\beta_m(u,v)$ and derivatives thereof. The inductive hypothesis for $\beta_{+m}$ is
\begin{align}
  \label{eq:127}
  \beta_{+m}(0)=\beta_0,\qquad \beta_{+m}'(0)=\beta_0',\qquad\sup_{[0,\varepsilon]}\left|\beta_{+m}''\right|\leq B.
\end{align}
From this we obtain
\begin{align}
  \label{eq:128}
  |\beta_{+m}'(v)-\beta'_0|\leq \left|\int_0^v\beta_{+m}''(v')dv'\right|\leq Bv.
\end{align}
In the following we use the notation $f(v)=\Landau_d(v^n)$ to denote
\begin{align}
  \label{eq:129}
  |f(v)|\leq C(d)v^n,
\end{align}
where $C(d)$ is a non-decreasing, continuous function of $d$. Using this, \eqref{eq:128} implies
\begin{align}
  \label{eq:130}
  \beta_{+m}'(v)=\beta_0'+\Landau_B(v).
\end{align}
From this we obtain
\begin{align}
  \label{eq:131}
  \beta_{+m}(v)=\beta_0+\beta'_0v+\Landau_B(v^2).
\end{align}
Now, for any function $f(v)=\Landau_d(v^n)$ we have
\begin{align}
  \label{eq:132}
  |f(v)|\leq C(d)v^n\leq Cv^{n-1},
\end{align}
provided we choose $\varepsilon$ sufficiently small depending on $d$, where on the right we have a fixed numerical constant. I.e.~$f(v)=\Landau(v^{n-1})$ (without an index on the Landau symbol). Hence, from \eqref{eq:131} we have
\begin{align}
  \label{eq:133}
  \beta_{+m}(v)=\beta_0+\Landau(v),
\end{align}
provided we choose $\varepsilon$ sufficiently small. In the following, we are going to use smallness conditions on $\varepsilon$ of this type without mentioning it any further. Since (see \eqref{eq:97})
\begin{align}
  \label{eq:134}
  \beta_m(u,v)=\beta_{+m}(v),\qquad\alpha_m(u,v)=\beta_{+m}(u),
\end{align}
we have
\begin{align}
  \label{eq:135}
  \beta_m(u,v)&=\beta_0+\beta_0'v+\Landau_B(v^2)=\beta_0+\Landau(v),\\
  \label{eq:136}
  \alpha_m(u,v)&=\beta_0+\beta_0'u+\Landau_B(u^2)=\beta_0+\Landau(v).
\end{align}
Taking partial derivatives of \eqref{eq:134} and using \eqref{eq:130} we obtain
\begin{align}
  \label{eq:137}
  \pp{\beta_m}{v}(u,v)=\beta_0'+\Landau_B(v),\qquad \pp{\alpha_m}{u}(u,v)=\beta_0'+\Landau_B(v).
\end{align}
The other first derivatives of $\alpha_m$, $\beta_m$ vanish. In particular, first order derivatives of $\alpha_m$, $\beta_m$ are bounded by a fixed constant, provided we choose $\varepsilon$ sufficiently small. Taking second derivatives of \eqref{eq:134} and using the third of \eqref{eq:127} we obtain
\begin{align}
  \label{eq:138}
  \left|\ppp{\beta_m}{v}(u,v)\right|,\left|\ppp{\alpha_m}{u}(u,v)\right|\leq B.
\end{align}
All other second derivatives of $\alpha_m$, $\beta_m$ vanish. Since
\begin{align}
  \label{eq:139}
  \alpha_{+m}(v)=\alpha_m(av,v)=\beta_{+m}(av),
\end{align}
we have, using \eqref{eq:127}, \eqref{eq:130}, \eqref{eq:131},
\begin{align}
  \label{eq:140}
  \alpha_{+m}(v)=\beta_0+a\beta_0'v+\Landau_B(v^2),\qquad \alpha_{+m}'(v)=a\beta_0'+\Landau_B(v),\qquad |\alpha_{+m}''(v)|\leq a^2B.
\end{align}

We derive properties of $x_m(u,v)$. The inductive hypothesis for $x_m$ is
\begin{align}
  \label{eq:141}
  x_m(0,0)=0,\qquad -\pp{x_m}{u}(0,0)=1&=\pp{x_m}{v}(0,0),\\
    \label{eq:142}
  \sup_{T_\varepsilon}\left|\ppp{x_m}{u}\right|,\sup_{T_\varepsilon}\left|\pppp{x_m}{u}{v}\right|,\sup_{T_\varepsilon}\left|\ppp{x_m}{v}\right|&\leq N_0.
\end{align}
$x_m(u,v)$ satisfies
\begin{align}
  \label{eq:143}
  \pp{x_m}{u}(u,v)&=\pp{x_m}{u}(0,0)+\int_0^u\left(\ppp{x_m}{u}+\pppp{x_m}{u}{v}\right)(u',u')du'+\int_u^v\pppp{x_m}{u}{v}(u,v')dv',\\
  \label{eq:144}
  \pp{x_m}{v}(u,v)&=\pp{x_m}{v}(0,0)+\int_0^u\left(\pppp{x_m}{u}{v}+\ppp{x_m}{v}\right)(u',u')du'+\int_u^v\ppp{x_m}{v}(u,v')dv'.
\end{align}
These equations correspond to integrating the second derivative of $x_m(u,v)$ from the origin $(0,0)$ along $u=v$ until $(u,u)$ and then integrating along an outgoing characteristic until $(u,v)$. Using the inductive hypothesis for $x_m$ we obtain
\begin{align}
  \label{eq:145}
  \pp{x_m}{u}(u,v)&=-1+\Landau_{N_0}(v),\\
  \label{eq:146}
  \pp{x_m}{v}(u,v)&=1+\Landau_{N_0}(v).
\end{align}
Due to $x_m(u,u)=0$ (see \eqref{eq:116}), we have
\begin{align}
  \label{eq:147}
  x_m(u,v)=\int_u^v\pp{x_m}{v}(u,v')dv'.
\end{align}
Using \eqref{eq:146} in \eqref{eq:147} we have
\begin{align}
  \label{eq:148}
  x_m(u,v)=v-u+\Landau_{N_0}(v^2)=\Landau(v).
\end{align}

We derive properties of $t_m(u,v)$. We have (see \eqref{eq:98})
\begin{align}
  \label{eq:149}
  t_m(u,v)=\int_0^u\left(\phi_m+\psi_m\right)(u',u')du'+\int_u^v\psi_m(u,v')dv', 
\end{align}
where
\begin{align}
  \label{eq:150}
  \phi_m=G(\alpha_m,\beta_m)\pp{x_m}{u},\qquad\psi_m=H(\alpha_m,\beta_m)\pp{x_m}{v},
\end{align}
and
\begin{align}
  \label{eq:151}
  G(\alpha,\beta)\defeq \frac{1}{\cin(\alpha,\beta)},\qquad H(\alpha,\beta)\defeq\frac{1}{\cout(\alpha,\beta)}.
\end{align}
We are going to use the notation
\begin{align}
  \label{eq:152}
  g_m(u,v)=G(\alpha_m(u,v),\beta_m(u,v)),\qquad h_m(u,v)=H(\alpha_m(u,v),\beta_m(u,v)).
\end{align}
$\cin(\alpha,\beta)$, $\cout(\alpha,\beta)$ are smooth functions of $\alpha$, $\beta$. In view of \eqref{eq:135}, \eqref{eq:136} provided we choose $\varepsilon$ sufficiently small, so are the functions $G(\alpha,\beta)$, $H(\alpha,\beta)$. Therefore,
\begin{align}
  \label{eq:153}
  g_m(u,v)&=g_0+\left(\pp{G}{\alpha}\right)_0\beta_0'u+\left(\pp{G}{\beta}\right)_0\beta_0'v+\Landau_B(v^2)=-\frac{1}{\eta_0}+\Landau(v),\\
  \label{eq:154}
  h_m(u,v)&=h_0+\left(\pp{H}{\alpha}\right)_0\beta_0'u+\left(\pp{H}{\beta}\right)_0\beta_0'v+\Landau_B(v^2)=\frac{1}{\eta_0}+\Landau(v).
\end{align}
Where we used
\begin{align}
  \label{eq:155}
  g_0=\frac{1}{\cin_0}=\frac{1}{\cin(\alpha_0,\beta_0)}=-\frac{1}{\eta_0},\qquad h_0=\frac{1}{\cout_0}=\frac{1}{\cout(\alpha_0,\beta_0)}=\frac{1}{\eta_0}.
\end{align}
We have
\begin{align}
  \label{eq:156}
  \pp{h_m}{u}=\pp{H}{\alpha}(\alpha_m,\beta_m)\pp{\alpha_m}{u},\qquad \pp{h_m}{v}=\pp{H}{\beta}(\alpha_m,\beta_m)\pp{\beta_m}{v},
\end{align}
which, through \eqref{eq:137} implies
\begin{align}
  \label{eq:157}
  \pp{h_m}{u}(u,v)=\left(\pp{H}{\alpha}\right)_0\beta_0'+\Landau_B(v),\qquad\pp{h_m}{v}(u,v)=\left(\pp{H}{\beta}\right)_0\beta_0'+\Landau_B(v).
\end{align}
We have
\begin{align}
  \label{eq:158}
  \ppp{h_m}{u}&=\ppp{H}{\alpha}(\alpha_m,\beta_m)\left(\pp{\alpha_m}{u}\right)^2+\pp{H}{\alpha}(\alpha_m,\beta_m)\ppp{\alpha_m}{u},\\
  \label{eq:159}
  \pppp{h_m}{u}{v}&=\pppp{H}{\alpha}{\beta}(\alpha_m,\beta_m)\pp{\alpha_m}{u}\pp{\beta_m}{v}.
\end{align}
Therefore, using \eqref{eq:138}, we have
\begin{align}
  \label{eq:160}
  \ppp{h_m}{u}(u,v)=\Landau_B(1),\qquad\pppp{h_m}{u}{v}(u,v)=\Landau(1).
\end{align}

We now derive properties of $t_m(u,v)$ using \eqref{eq:149}. Using \eqref{eq:153}, \eqref{eq:154} together with \eqref{eq:145}, \eqref{eq:146} we obtain
\begin{align}
  \label{eq:161}
  \phi_m(u,v)=\frac{1}{\eta_0}+\Landau_{N_0}(v),\qquad\psi_m(u,v)=\frac{1}{\eta_0}+\Landau_{N_0}(v).
\end{align}
Using these in \eqref{eq:149} we obtain
\begin{align}
  \label{eq:162}
  t_m(u,v)&=\frac{1}{\eta_0}(u+v)+\Landau_{N_0}(v^2)=\Landau(v),\\
  \label{eq:163}
  \pp{t_m}{v}(u,v)&=\psi_m(u,v)=\frac{1}{\eta_0}+\Landau_{N_0}(v).
\end{align}
From \eqref{eq:149} we have
\begin{align}
  \label{eq:164}
  \pp{t_m}{u}(u,v)=\phi_m(u,u)+\int_u^v\pp{\psi_m}{u}(u,v')dv'.
\end{align}
We have
\begin{align}
  \label{eq:165}
  \pp{\psi_m}{u}=\pp{h_m}{u}\pp{x_m}{v}+h_m\pppp{x_m}{u}{v}.
\end{align}
Using \eqref{eq:154} and the first of \eqref{eq:157} together with \eqref{eq:142}, \eqref{eq:146} we obtain
\begin{align}
  \label{eq:166}
  \pp{\psi_m}{u}(u,v)=\left(\pp{H}{\alpha}\right)_0\beta_0'+\Landau_B(v)+\Landau_{N_0}(1)=\Landau_{N_0}(1).
\end{align}
Using this together with the first of \eqref{eq:161} in \eqref{eq:164} we obtain
\begin{align}
  \label{eq:167}
  \pp{t_m}{u}(u,v)=\frac{1}{\eta_0}+\Landau_{N_0}(v).
\end{align}
We turn to estimates for the second partial derivatives of $t(u,v)$. Since (see \eqref{eq:164})
\begin{align}
  \label{eq:168}
  \pppp{t_m}{u}{v}(u,v)=\pp{\psi_m}{u}(u,v),
\end{align}
we deduce from \eqref{eq:166}
\begin{align}
  \label{eq:169}
  \pppp{t_m}{u}{v}(u,v)=\Landau_{N_0}(1).
\end{align}
Analogous to the expression for $\pp{\psi_m}{u}(u,v)$, we have
\begin{align}
  \label{eq:170}
  \pp{\psi_m}{v}(u,v)=\Landau_{N_0}(1),\qquad  \pp{\phi_m}{u}(u,v)=\Landau_{N_0}(1),\qquad\pp{\phi_m}{v}(u,v)=\Landau_{N_0}(1),
\end{align}
the first of which implies
\begin{align}
  \label{eq:171}
  \ppp{t_m}{v}(u,v)=\pp{\psi_m}{v}(u,v)=\Landau_{N_0}(1).
\end{align}
From \eqref{eq:164} we have
\begin{align}
  \label{eq:172}
  \ppp{t_m}{u}(u,v)=\pp{\phi_m}{u}(u,u)+\pp{\phi_m}{v}(u,u)+\pp{}{u}\left(\int_u^v\pp{\psi_m}{u}(u,v')dv'\right).
\end{align}
For the last term in \eqref{eq:172} we use
\begin{align}
  \label{eq:173}
  \int_u^v\pp{\psi_m}{u}(u,v')dv'&=\int_u^v\left(\pp{h_m}{u}\pp{x_m}{v}+h_m\pppp{x_m}{u}{v}\right)(u,v')dv'\nonumber\\
                               &=\int_u^v\left(\pp{h_m}{u}\pp{x_m}{v}\right)(u,v')dv'\nonumber\\
  &\quad+\left(h_m\pp{x_m}{u}\right)(u,v)-\left(h_m\pp{x_m}{u}\right)(u,u)-\int_u^v\left(\pp{h_m}{v}\pp{x_m}{u}\right)(u,v')dv',
\end{align}
where we integrated by parts for the second term in the bracket in the first line. Using this, the last term in \eqref{eq:172} becomes
\begin{align}
  \label{eq:174}
  \pp{}{u}\left(\int_u^v\pp{\psi_m}{u}(u,v')dv'\right)&=\int_u^v\left(\ppp{h_m}{u}\pp{x_m}{v}+\pp{h_m}{u}\pppp{x_m}{u}{v}\right)(u,v')dv'-\left(\pp{h_m}{u}\pp{x_m}{v}\right)(u,u)\nonumber\\
                                                    &\qquad+\left(\pp{h_m}{u}\pp{x_m}{u}\right)(u,v)+\left(h_m\ppp{x_m}{u}\right)(u,v)\nonumber\\
                                                    &\qquad-\left(\pp{h_m}{u}\pp{x_m}{u}\right)(u,u)-\left(h_m\ppp{x_m}{u}\right)(u,u)\nonumber\\
                                                    &\qquad-\left(\pp{h_m}{v}\pp{x_m}{u}\right)(u,u)-\left(h_m\pppp{x_m}{u}{v}\right)(u,u)\nonumber\\
                                                      &\qquad -\int_u^v\left(\pppp{h_m}{u}{v}\pp{x_m}{u}+\pp{h_m}{v}\ppp{x_m}{u}\right)(u,v')dv'\nonumber\\
&\hspace{50mm}  +\left(\pp{h_m}{v}\pp{x_m}{u}\right)(u,u).
\end{align}
(We don't carry out the cancellation in the last two lines in favour of readability). Using \eqref{eq:157}, \eqref{eq:160} together with \eqref{eq:142}, \eqref{eq:145}, \eqref{eq:146} we obtain
\begin{align}
  \label{eq:175}
  \pp{}{u}\left(\int_u^v\pp{\psi_m}{u}(u,v')dv'\right)=\Landau_{N_0}(1).
\end{align}
Using this in turn in \eqref{eq:172}, together with the expressions for the partial derivatives of $\phi_m$ as given in \eqref{eq:170}, we obtain
\begin{align}
  \label{eq:176}
  \ppp{t_m}{u}(u,v)=\Landau_{N_0}(1).
\end{align}
We summarize the properties of $t_m(u,v)$:
\begin{align}
  \label{eq:177}
  t_m(u,v)&=\frac{1}{\eta_0}(u+v)+\Landau_{N_0}(v^2)=\Landau(v),\\  
  \label{eq:178}
  \pp{t_m}{u}(u,v)&=\frac{1}{\eta_0}+\Landau_{N_0}(v),\qquad   \pp{t_m}{v}(u,v)=\frac{1}{\eta_0}+\Landau_{N_0}(v),\\
  \label{eq:179}
  \ppp{t_m}{u}(u,v)&=\Landau_{N_0}(1),\qquad  \pppp{t_m}{u}{v}(u,v)=\Landau_{N_0}(1),\qquad  \ppp{t_m}{v}(u,v)=\Landau_{N_0}(1). 
\end{align}

We now derive properties of $t_+(v)$, $x_+(v)$. Using \eqref{eq:142}, \eqref{eq:145}, \eqref{eq:146}, \eqref{eq:148} we obtain the following properties of $x_{+m}(v)$:
\begin{align}
  \label{eq:180}
  x_{+m}(v)&=x_m(av,v)=(1-a)v+\Landau_{N_0}(v^2),\\
  \label{eq:181}
  x_{+m}'(v)&=a\pp{x_m}{u}(av,v)+\pp{x_m}{v}(av,v)=1-a+\Landau_{N_0}(v),\\
  \label{eq:182}
  x_{+m}''(v)&=a^2\ppp{x_m}{u}(av,v)+2a\pppp{x_m}{u}{v}(av,v)+\ppp{x_m}{v}(av,v)=\Landau_{N_0}(1).
\end{align}
Similarly, but using \eqref{eq:177}, \eqref{eq:178}, \eqref{eq:179} we obtain the following properties of $t_{+m}(v)$:
\begin{align}
  \label{eq:183}
  t_{+m}(v)&=t_m(av,v)=\frac{1+a}{\eta_0}v+\Landau_{N_0}(v^2),\\
  \label{eq:184}
  t_{+m}'(v)&=a\pp{t_m}{u}(av,v)+\pp{t_m}{v}(av,v)=\frac{1+a}{\eta_0}+\Landau_{N_0}(v).\\
  \label{eq:185}
  t_{+m}''(v)&=a^2\ppp{t_m}{u}(av,v)+2a\pppp{t_m}{u}{v}(av,v)+\ppp{t_m}{v}(av,v)=\Landau_{N_0}(1).
\end{align}

Let $t\mapsto x_{0}^\ast(t)$ describe the outgoing characteristic in the state ahead, starting at the origin (this coincides with the left boundary  $\mathcal{B}$ of the future development of the data, see figures \ref{future_development}, \ref{reflection_problem}). We use the notation
\begin{align}
  \label{eq:186}
  x_{0,m}^\ast(v)=x_{0}^\ast(t_{+m}(v))
\end{align}
We have
\begin{align}
  \label{eq:187}
  \frac{dx_{0,m}^\ast}{dv}(v)=\frac{dx_{0}^\ast}{dt}\Big(t_{+m}(v)\Big)t_{+m}'(v).
\end{align}
Let us denote by $\cout^\ast_0(t)$ the characteristic speed of the outgoing characteristic starting at the origin, i.e.
\begin{align}
  \label{eq:188}
  \frac{dx_0^\ast}{dt}(t)=\cout^\ast_0(t).
\end{align}
$\cout_0^\ast(t)$ being a smooth function of $t$, we have
\begin{align}
  \label{eq:189}
  \cout_0^\ast(t)=(\cout_0^\ast)_0+\Landau(t),
\end{align}
therefore (see \eqref{eq:183}),
\begin{align}
  \label{eq:190}
  \cout_0^\ast(t_{+m}(v))=(\cout_0^\ast)_0+\Landau(v).
\end{align}
Using this together with \eqref{eq:184} in \eqref{eq:187} we obtain
\begin{align}
  \label{eq:191}
  \frac{dx_{0,m}^\ast}{dv}(v)=(\cout_0^\ast)_0\frac{1+a}{\eta_0}+\Landau_{N_0}(v).
\end{align}
Using this and \eqref{eq:181} we have (see also \eqref{eq:33})
\begin{align}
  \label{eq:192}
  x_{+m}'(v)-\frac{dx_{0,m}^\ast}{dv}(v)=\frac{1-a}{V_0}\Big(V_0-(\cout_0^\ast)_0\Big)+\Landau_{N_0}(v).
\end{align}
Hence
\begin{align}
  \label{eq:193}
  x_{+m}(v)-x_{0,m}^\ast(v)=\frac{1-a}{V_0}\Big(V_0-(\cout_0^\ast)_0\Big)v+\Landau_{N_0}(v^2).
\end{align}
By the determinism condition \eqref{eq:15} and \eqref{eq:17}, the factor in front of $v$ in the first term is strictly positive. By choosing $\varepsilon$ sufficiently small, the remainder in \eqref{eq:193} in absolute value can be made less or equal to $\frac{1-a}{V_0}(V_0-(\cout_0^\ast)_0)v$. Hence we obtain
\begin{align}
  \label{eq:194}
  x_{+m}(v)-x_{0,m}^\ast(v)\geq 0
\end{align}
for $v\in[0,\varepsilon]$. I.e.~the curve $v\mapsto (t_{+m}(v),x_{+m}(v))$ lies in the domain of the future development (see figure \ref{reflection_problem}).

We derive properties of $\alpha_{-m}(v)$, $\beta_{-m}(v)$. We have
\begin{align}
  \label{eq:195}
  \alpha_{-m}(v)=\alpha^\ast(t_{+m}(v),x_{+m}(v)),\qquad \beta_{-m}(v)=\beta^\ast(t_{+m}(v),x_{+m}(v)),
\end{align}
where $\alpha^\ast(t,x)$, $\beta^\ast(t,x)$ are smooth functions of their arguments, given by the solution in the future development in the state ahead. The second derivative of $\alpha_{-m}(v)$ is given by
\begin{align}
  \label{eq:196}
  \alpha_{-m}''(v)&=\pp{\alpha^\ast}{t}(t_{+m}(v),x_{+m}(v))t_{+m}''(v)+\pp{\alpha^\ast}{x}(t_{+m}(v),x_{+m}(v))x_{+m}''(v)\nonumber\\
               &\qquad+\ppp{\alpha^\ast}{t}(t_{+m}(v),x_{+m}(v))\left(t_{+m}'(v)\right)^2+2\pppp{\alpha^\ast}{t}{x}(t_{+m}(v),x_{+m}(v))t_{+m}'(v)x_{+m}'(v)\nonumber\\
  &\qquad\qquad+\ppp{\alpha^\ast}{x}(t_{+m}(v),x_{+m}(v))\left(x_{+m}'(v)\right)^2.
\end{align}
Making use of \eqref{eq:180},\ldots,\eqref{eq:185} we deduce
\begin{align}
  \label{eq:197}
  \alpha_{-m}''(v)=\Landau_{N_0}(1).
\end{align}
Together with (see \eqref{eq:181}, \eqref{eq:184})
\begin{align}
  \label{eq:198}
  \alpha_{-m}'(0)&=\left(\pp{\alpha^\ast}{t}\right)_0t_{+m}'(0)+\left(\pp{\alpha^\ast}{x}\right)_0x_{+m}'(0)\nonumber\\
              &=\left(\pp{\alpha^\ast}{t}\right)_0\frac{1}{\eta_0}(1+a)+\left(\pp{\alpha^\ast}{x}\right)_0(1-a)
\end{align}
we obtain
\begin{align}
  \label{eq:199}
  \alpha_{-m}'(v)=\left(\pp{\alpha^\ast}{t}\right)_0\frac{1}{\eta_0}(1+a)+\left(\pp{\alpha^\ast}{x}\right)_0(1-a)+\Landau_{N_0}(v).
\end{align}
With $\alpha_{-m}(0)=\alpha^\ast_0$, this implies
\begin{align}
  \label{eq:200}
  \alpha_{-m}(v)&=\alpha^\ast_0+\left(\pp{\alpha^\ast}{t}\right)_0\frac{1}{\eta_0}(1+a)v+\left(\pp{\alpha^\ast}{x}\right)_0(1-a)v+\Landau_{N_0}(v^2).
\end{align}
We note that analogous estimates hold for $\beta_{-m}(v)$, $\beta_{-m}'(v)$, $\beta_{-m}''(v)$.

We derive properties of $V_m(v)$. $V_m(v)$ is given by
\begin{align}
  \label{eq:201}
  V_m(v)=\frac{\jump{\rho_m(v)w_m(v)}}{\jump{\rho_m(v)}}.
\end{align}
Here
\begin{align}
  \label{eq:202}
  \jump{\rho_m(v)}&=\rho_{+m}(v)-\rho_{-m}(v)\nonumber\\
                &=\rho(\alpha_{+m}(v),\beta_{+m}(v))-\rho(\alpha_{-m}(v),\beta_{-m}(v))
\end{align}
and $\rho(\alpha,\beta)$ is a given smooth function of its arguments. Similarly for $\jump{\rho_m(v)w_m(v)}$. Using the properties for $\beta_{+m}(v)$ as given by \eqref{eq:127}, \eqref{eq:130}, the properties of $\alpha_{+m}(v)$, as given by \eqref{eq:140}, the properties of $\alpha_{-m}(v)$ (and $\beta_{-m}(v)$) as given by \eqref{eq:199}, \eqref{eq:200} (and analogous estimates with $\beta_{-m}$ in the role of $\alpha_{-m}$), we obtain
\begin{align}
  \label{eq:203}
  V_m(v)=V_0+\Landau(v),\qquad |V_m'(v)|\leq C.
\end{align}

We derive properties of $\cin_m(u,v)$, $\cout_m(u,v)$. We recall the notation
\begin{align}
  \label{eq:204}
  \cin_m(u,v)=\cin(\alpha_m(u,v),\beta_m(u,v)),\qquad  \cout_m(u,v)=\cout(\alpha_m(u,v),\beta_m(u,v)).
\end{align}
$\cin(\alpha,\beta)$, $\cout(\alpha,\beta)$ being smooth functions, the argument is analogous to the argument used to derive properties of $g_m(u,v)$, $h_m(u,v)$ (see \eqref{eq:151}, \eqref{eq:152}). Analogous to \eqref{eq:153}, \eqref{eq:154} we have
\begin{align}
  \label{eq:205}
  \cin_m(u,v)&=(\cin)_0+\left(\pp{\cin}{\alpha}\right)_0\beta_0'u+\left(\pp{\cin}{\beta}\right)_0\beta_0'v+\Landau_B(v^2)=-\eta_0+\Landau(v),\\
  \label{eq:206}
  \cout_m(u,v)&=(\cout)_0+\left(\pp{\cout}{\alpha}\right)_0\beta_0'u+\left(\pp{\cout}{\beta}\right)_0\beta_0'v+\Landau_B(v^2)=\eta_0+\Landau(v).
\end{align}
Analogous to \eqref{eq:157} we have
\begin{align}
  \label{eq:207}
  \pp{\cin_m}{u}(u,v)=\left(\pp{\cin}{\alpha}\right)_0\beta_0'+\Landau_B(v),\qquad\pp{\cin_m}{v}(u,v)=\left(\pp{\cin}{\beta}\right)_0\beta_0'+\Landau_B(v).
\end{align}
Analogous to \eqref{eq:160} we have
\begin{align}
  \label{eq:208}
\ppp{\cin_m}{u}(u,v)=\Landau_B(1),\qquad\ppp{\cin_m}{v}(u,v)=\Landau_B(1),\qquad\pppp{\cin_m}{u}{v}(u,v)=\Landau(1).
\end{align}
We note that the analogous estimates hold for $\cout_m(u,v)$ and derivatives thereof.

We derive properties of
\begin{align}
  \label{eq:209}
  \cin_{+m}(v)=\cin(\alpha_{+m}(v),\beta_{+m}(v)),\qquad\cout_{+m}(v)=\cout(\alpha_{+m}(v),\beta_{+m}(v)).
\end{align}
Using the properties for $\beta_{+m}(v)$ as given by \eqref{eq:127}, \eqref{eq:130}, \eqref{eq:131} and the properties for $\alpha_{+m}(v)$ as given by \eqref{eq:140}, we obtain
\begin{alignat}{3}
  \label{eq:210}
  \cin_{+m}(v)&=-\eta_0+\Landau(v),&\qquad \cout_{+m}(v)&=\eta_0+\Landau(v),\\
  \label{eq:211}
  \cin_{+m}'(v)&=\Landau(1),&\cout_{+m}'(v)&=\Landau(1).
\end{alignat}
Using these together with \eqref{eq:203} in
\begin{align}
  \label{eq:212}
      \Gamma_m(v)=a\frac{\cout_{+m}(v)}{\cin_{+m}(v)}\,\frac{V_m(v)-\cin_{+m}(v)}{\cout_{+m}(v)-V_m(v)},
\end{align}
we obtain
\begin{align}
  \label{eq:213}
  \Gamma_m(v)=-1+\Landau(v),\qquad |\Gamma_m'(v)|\leq C,
\end{align}
where for the constant $-1$ in the first equation we also used \eqref{eq:33}.

We derive properties of $M_m(u,v)$. We have (see \eqref{eq:112})
\begin{align}
  \label{eq:214}
  M_m(u,v)=\mu_m(u,v)\pp{x_m}{u}(u,v)+\nu_m(u,v)\pp{x_m}{v}(u,v).
\end{align}
Here (see \eqref{eq:79})
\begin{align}
  \label{eq:215}
  \mu_m=\frac{1}{\cout_m-\cin_m}\frac{\cout_m}{\cin_m}\pp{\cin_m}{v},\qquad \nu_m=-\frac{1}{\cout_m-\cin_m}\frac{\cin_m}{\cout_m}\pp{\cout_m}{u}.
\end{align}
Using the properties of $\cin_m$, as given by \eqref{eq:205}, \eqref{eq:207}, \eqref{eq:208} (and analogous equations for $\cout_m$), we obtain
\begin{align}
  \label{eq:216}
  \mu_m(u,v)=-\frac{1}{2\eta_0}\left(\pp{\cin}{\beta}\right)_0\beta_0'+\Landau_B(v),\qquad\nu_m(u,v)=\frac{1}{2\eta_0}\left(\pp{\cout}{\alpha}\right)_0\beta_0'+\Landau_B(v),\\
  \label{eq:217}
  \pp{\mu_m}{u}(u,v)=\Landau(1),\quad\pp{\mu_m}{v}(u,v)=\Landau_B(1),\quad\pp{\nu_m}{u}(u,v)=\Landau_B(1),\quad\pp{\nu_m}{v}(u,v)=\Landau(1).
\end{align}
Together with \eqref{eq:142}, \eqref{eq:145}, \eqref{eq:146} we obtain
\begin{align}
  \label{eq:218}
  M_m(u,v)&=\frac{1}{2\eta_0}\left(\left(\pp{\cin}{\beta}\right)_0+\left(\pp{\cout}{\alpha}\right)_0\right)\beta_0'+\Landau_B(v)+\Landau_{N_0}(v),\\
  \label{eq:219}
  \pp{M_m}{u}(u,v)&=\Landau_B(1)+\Landau_{N_0}(1),\qquad\pp{M_m}{v}(u,v)= \Landau_B(1)+\Landau_{N_0}(1).
\end{align}

We turn to the properties of $x_{m+1}(u,v)$. We have
\begin{align}
  \label{eq:220}
  x_{m+1}(u,v)=v-u+\int_u^v\Phi_m(u,v')dv',
\end{align}
where
\begin{align}
  \label{eq:221}
  \Phi_{m}(u,v)=\int_0^v\Lambda_m(u')du'+\int_{av}^uM_m(u',v)du',
\end{align}
where $M_m(u,v)$ is given in \eqref{eq:214} and
\begin{align}
  \label{eq:222}
  \Lambda_m(u)&=\Gamma_m(u)a\ppp{x_m}{u}(au,u)+\Gamma_m'(u)\pp{x_m}{u}(au,u)+\Gamma_m(u)M_m(au,u).
\end{align}
The first order derivatives are
\begin{align}
  \label{eq:223}
  \pp{x_{m+1}}{u}(u,v)&=-1-\Phi_{m}(u,u)+\int_u^v\pp{\Phi_{m}}{u}(u,v')dv'\nonumber\\
                      &=-1-\int_0^u\Lambda_m(u')du'-\int_{au}^uM_m(u',u)du'+\int_u^vM_m(u,v')dv',\\
  \label{eq:224}
  \pp{x_{m+1}}{v}(u,v)&=1+\Phi_{m}(u,v)\nonumber\\
                      &=1+\int_0^v\Lambda_m(u')du'+\int_{av}^uM_m(u',v)du'.
\end{align}
From \eqref{eq:223}, \eqref{eq:224} we see that
\begin{align}
  \label{eq:225}
  -\pp{x_{m+1}}{u}(0,0)=1=\pp{x_{m+1}}{v}(0,0).
\end{align}
Taking another derivative of \eqref{eq:223}, \eqref{eq:224}, we obtain
\begin{align}
  \label{eq:226}
  \pppp{x_{m+1}}{u}{v}(u,v)&=M_m(u,v),\\
  \label{eq:227}
  \ppp{x_{m+1}}{u}(u,v)&=-\Lambda_m(u)-2M_m(u,u)+aM_m(au,u)\nonumber\\
  &\qquad-\int_{au}^u\pp{M_m}{v}(u',u)du'+\int_u^v\pp{M_m}{u}(u,v')dv',\\
  \label{eq:228}
  \ppp{x_{m+1}}{v}(u,v)&=\Lambda_m(v)-aM_m(av,v)+\int_{av}^u\pp{M_m}{v}(u',v)du'.
\end{align}
Writing out the expressions for $\Lambda_m(u)$ and $\Lambda_m(v)$, the second and third of these are
\begin{align}
  \label{eq:229}
  \ppp{x_{m+1}}{u}(u,v)&=-\Gamma_{m}(u)a\ppp{x_m}{u}(au,u)-\Gamma_{m}'(u)\pp{x_{m}}{u}(au,u)-\Gamma_{m}(u)M_m(au,u)\nonumber\\
                           &\qquad +aM_m(au,u) -2M_m(u,u)\nonumber\\
  &\qquad-\int_{au}^u\pp{M_m}{v}(u',u)du'+\int_u^v\pp{M_m}{u}(u,v')dv',\\
  \ppp{x_{m+1}}{v}(u,v)&=\Gamma_{m}(v)a\ppp{x_{m}}{u}(av,v)+\Gamma_{m}'(v)\pp{x_{m}}{u}(av,v)+\Gamma_{m}(v)M_m(av,v)\nonumber\\
  \label{eq:230}
                       &\qquad -aM_m(av,v)+\int_{av}^u\pp{M_m}{v}(u',v)du'.
\end{align}
Using \eqref{eq:218} we have
\begin{align}
  \label{eq:231}
  \left|\pppp{x_{m+1}}{u}{v}(u,v)\right|\leq C.
\end{align}
Using  \eqref{eq:145}, \eqref{eq:213}, \eqref{eq:218}, \eqref{eq:219} and the induction hypothesis for $x_m$ we obtain
\begin{align}
  \label{eq:232}
  \left|\ppp{x_{m+1}}{u}(u,v)\right|\leq (1+\Landau(u))aN_0+C.
\end{align}
Therefore, choosing $\varepsilon$ sufficiently small, we obtain
\begin{align}
  \label{eq:233}
  \left|\ppp{x_{m+1}}{u}(u,v)\right|\leq C+aN_0.
\end{align}
Analogously we obtain
\begin{align}
  \label{eq:234}
  \left|\ppp{x_{m+1}}{v}(u,v)\right|\leq C+aN_0.
\end{align}
Taking the supremum in $T_\varepsilon$ of \eqref{eq:231}, \eqref{eq:233}, \eqref{eq:234} we obtain
\begin{align}
  \label{eq:235}
  \sup_{T_\varepsilon}\left|\pppp{x_{m+1}}{u}{v}\right|\leq C,\qquad \sup_{T_\varepsilon}\left|\ppp{x_{m+1}}{u}\right|\leq C+aN_0,\qquad \sup_{T_\varepsilon}\left|\ppp{x_{m+1}}{v}\right|\leq C+aN_0.
\end{align}
Therefore,
\begin{align}
  \label{eq:236}
  \|x_{m+1}\|_2=\max\left\{\sup_{T_\varepsilon}\left|\ppp{x_{m+1}}{u}\right|,\sup_{T_\varepsilon}\left|\pppp{x_{m+1}}{u}{v}\right|,\sup_{T_\varepsilon}\left|\ppp{x_{m+1}}{v}\right|\right\}\leq C+aN_0.
\end{align}
Choosing the constant $N_0$ sufficiently large, such that
\begin{align}
  \label{eq:237}
  \frac{C}{1-a}\leq N_0,
\end{align}
where the constant $C$ on the left is the constant $C$ appearing in \eqref{eq:236}, we obtain
\begin{align}
  \label{eq:238}
  \|x_{m+1}\|_2\leq N_0,
\end{align}
Together with  \eqref{eq:220} (which implies $x_{m+1}(0,0)=0$) and \eqref{eq:225} we see that the function $x_{m+1}(u,v)$ lies in $B_{N_0}$ (compare the definition of $B_{N_0}$ in \eqref{eq:119}).

We derive properties of $\beta_{+,m+1}(v)$. We have
\begin{align}
  \label{eq:239}
  \beta_{+,m+1}(v)=H\Big(\alpha_{+m}(v),\alpha_{-m}(v),\beta_{-m}(v)\Big).
\end{align}
In view of (see \eqref{eq:140}, \eqref{eq:200} (and analogous expressions for $\beta_{-m}$) in conjunction with \eqref{eq:47})
\begin{align}
  \label{eq:240}
  \alpha_{+m}(0)=\beta_0,\qquad\alpha_{-m}(0)=\alpha_{-0},\qquad \beta_{-m}(0)=\beta_{-0},
\end{align}
we have (see \eqref{eq:54})
\begin{align}
  \label{eq:241}
  \beta_{+,m+1}(0)=\beta_0.
\end{align}
For the first derivative we have
\begin{align}
  \label{eq:242}
    \beta_{+,m+1}'(v)&=F\Big(\alpha_{+m}(v),\alpha_{-m}(v),\beta_{-m}(v)\Big)\alpha_{+m}'(v)\nonumber\\
             &\quad+M_1\Big(\alpha_{+m}(v),\alpha_{-m}(v),\beta_{-m}(v)\Big)\beta_{-m}'(v)\nonumber\\
             &\quad+M_2\Big(\alpha_{+m}(v),\alpha_{-m}(v),\beta_{-m}(v)\Big)\alpha_{-m}'(v),
\end{align}
Hence, (see \eqref{eq:140}, \eqref{eq:199} (and the analogous for $\beta_{-m}'(v)$) in conjunction with, \eqref{eq:66}, \eqref{eq:67} and \eqref{eq:61})
\begin{align}
  \label{eq:243}
  \beta_{+,m+1}'(0)=F_0a\beta_0'+M_1\Big(\alpha_0,\alpha_{-0},\beta_{-0}\Big)\beta_-'(0)+M_2\Big(\alpha_0,\alpha_{-0},\beta_{-0}\Big)\alpha_-'(0).
\end{align}
In view of the expression for $\beta_0'$ as given in \eqref{eq:72} we obtain
\begin{align}
  \label{eq:244}
  \beta_{+,m+1}'(0)=\beta_0'.
\end{align}
For the second derivative we have
\begin{align}
  \label{eq:245}
  \beta_{+,m+1}''(v)&=F(\alpha_{+m}(v),\alpha_{-m}(v),\beta_{-m}(v))\alpha_{+m}''(v)\nonumber\\
                    &\qquad +M_1(\alpha_{+m}(v),\alpha_{-m}(v),\beta_{-m}(v))\beta_{-m}''(v)\nonumber\\
                    &\qquad +M_2(\alpha_{+m}(v),\alpha_{-m}(v),\beta_{-m}(v))\alpha_{-m}''(v)\nonumber\\
                    &\qquad +G(\alpha_{+m}(v),\alpha_{-m}(v),\beta_{-m}(v),\alpha_{+m}'(v),\alpha_{-m}'(v),\beta_{-m}'(v)).
\end{align}
Here $F$, $M_1$, $M_1$, $G$ are smooth functions of their arguments. Furthermore, (see \eqref{eq:61}),
\begin{align}
  \label{eq:246}
  F(\alpha_{+m}(0),\alpha_{-m}(0),\beta_{-m}(0))=-a^2.
\end{align}
Using the properties of $\alpha_{-m}$ as given by \eqref{eq:197}, \eqref{eq:199}, \eqref{eq:200} and analogous estimates but with $\beta_{-m}$ in the role of $\alpha_{-m}$, as well as the properties of $\alpha_{+m}$ as given by \eqref{eq:140} we obtain
\begin{align}
  \label{eq:247}
  \left|\beta_{+,m+1}''(v)\right|\leq (a^2+Cv)a^2B+\Landau_{N_0}(1).
\end{align}
Taking into account our above choice of $N_0$, which makes $N_0$ into a fixed numerical constant, we have
\begin{align}
  \label{eq:248}
  \left|\beta_{+,m+1}''(v)\right|\leq a^4B+C.
\end{align}
Taking the supremum in $[0,\varepsilon]$ we have (see \eqref{eq:120} for the definition of $\|f\|_1$)
\begin{align}
  \label{eq:249}
  \|\beta_{+,m+1}\|_1\leq a^4B+C.
\end{align}
Choosing the constant $B$ sufficiently large, such that
\begin{align}
  \label{eq:250}
  \frac{C}{1-a^4}\leq B,
\end{align}
where the constant $C$ is the one appearing in \eqref{eq:249}, we arrive at
\begin{align}
  \label{eq:251}
  \|\beta_{+,m+1}\|_1\leq B.
\end{align}
Together with \eqref{eq:241}, \eqref{eq:244} we conclude that $\beta_{+,m+1}\in B_B$ (see \eqref{eq:118} for the definition of $B_B$). This completes the proof of the inductive step.
\end{proof}

\subsection{Convergence}
\begin{proposition}
For $\varepsilon$ sufficiently small, the sequence
  \begin{align}
\label{eq:252}
    ((\beta_{+m},x_m);m=0,1,2,\ldots)   
  \end{align}
converges in $B_B\times B_{N_0}$.
\end{proposition}
\begin{proof}
We use the notation
\begin{align}
  \label{eq:253}
  \Delta_mf\defeq f_m-f_{m-1}.
\end{align}
We first look at differences of $\alpha$ and $\beta$. Using the leading order behavior of $\beta_{+m}(v)$ as given in \eqref{eq:130}, \eqref{eq:131} and recalling the definition of $\|f\|_1$ in \eqref{eq:120}, we have
\begin{align}
  \label{eq:254}
  |\Delta_m\beta_+(v)|\leq v^2\|\Delta_m\beta_{+}\|_1,\qquad
  |\Delta_m\beta_+'(v)|\leq v\|\Delta_m\beta_{+}\|_1,
\end{align}
while from the leading order behavior of $\alpha_{+m}$ as given in \eqref{eq:140} we have
\begin{align}
  \label{eq:255}
  |\Delta_m\alpha_+|\leq v^2\|\Delta_m\beta_{+}\|_1,\qquad|\Delta_m\alpha_+'|\leq av\|\Delta_m\beta_{+}\|_1,\qquad|\Delta_m\alpha_+''|\leq a^2\|\Delta_m\beta_{+}\|_1.
\end{align}
Using $\beta_m(u,v)=\beta_{+m}(v)$, $\alpha_m(u,v)=\beta_{+m}(u)$ we deduce
\begin{align}
  \label{eq:256}
  \Delta_m\pppp{\alpha}{u}{v}&=\Delta_m\ppp{\alpha}{v}=0,\qquad\left|\Delta_m\ppp{\alpha}{u}\right|\leq\|\Delta_m\beta_{+}\|_1,\\
  \label{eq:257}
  \Delta_m\pppp{\beta}{u}{v}&=\Delta_m\ppp{\beta}{u}=0,\qquad\left|\Delta_m\ppp{\beta}{v}\right|\leq\|\Delta_m\beta_{+}\|_1.
\end{align}

We turn to estimates of $\Delta_mx(u,v)$ and derivatives thereof. From \eqref{eq:143}, \eqref{eq:144} together with \eqref{eq:141} we have
\begin{align}
  \label{eq:258}
  \left|\Delta_m\pp{x}{u}\right|&\leq v\left(\sup_{T_\varepsilon}\left|\Delta_m\ppp{x}{u}\right|+\sup_{T_\varepsilon}\left|\Delta_m\pppp{x}{u}{v}\right|\right),\\
  \label{eq:259}
  \left|\Delta_m\pp{x}{v}\right|&\leq v\left(\sup_{T_\varepsilon}\left|\Delta_m\ppp{x}{v}\right|+\sup_{T_\varepsilon}\left|\Delta_m\pppp{x}{u}{v}\right|\right).
\end{align}
Using these in \eqref{eq:147} we obtain
\begin{align}
  \label{eq:260}
  |\Delta_mx|\leq v^2\left(\sup_{T_\varepsilon}\left|\Delta_m\ppp{x}{v}\right|+\sup_{T_\varepsilon}\left|\Delta_m\pppp{x}{u}{v}\right|\right).
\end{align}

We look at $\Delta_mg$, $\Delta_mh$. We recall that
\begin{align}
  \label{eq:261}
  g_m(u,v)=G(\alpha_m(u,v),\beta_m(u,v)),\qquad h_m(u,v)=H(\alpha_m(u,v),\beta_m(u,v)),
\end{align}
where $G(\alpha,\beta)$, $H(\alpha,\beta)$ are smooth functions of their arguments (see \eqref{eq:151} for the definition of those functions). In view of \eqref{eq:131}, \eqref{eq:140}, the points $(\alpha_{+m}(v),\beta_{+m}(v))$ and $(\alpha_{+m-1}(v),\beta_{+m-1}(v))$ lie in a ball in $\mathbb{R}^2$ centered at $(\alpha_0,\beta_0)$. Hence so does the line segment joining them. We therefore have
\begin{align}
  \label{eq:262}
  |\Delta_mg|&\leq Cv^2\|\Delta_m\beta_{+}\|_1,\\
  \label{eq:263}
  \left|\Delta_m\pp{g}{u}\right|,\left|\Delta_m\pp{g}{v}\right|&\leq Cv\|\Delta_m\beta_{+}\|_1,\\
  \label{eq:264}
  \left|\Delta_m\ppp{g}{u}\right|,\left|\Delta_m\ppp{g}{v}\right|&\leq C\|\Delta_m\beta_{+}\|_1,\qquad \left|\Delta_m\pppp{g}{u}{v}\right|\leq Cv\|\Delta_m\beta_{+}\|_1.
\end{align}
The different behavior of the mixed derivative is due to $\pppp{\alpha_m}{u}{v}=0=\pppp{\beta_m}{u}{v}$. The same estimates hold with $h$ in the role of $g$.

We look at $\phi$, $\psi$. We recall
\begin{align}
  \label{eq:265}
  \phi_m(u,v)=g_m(u,v)\pp{x_m}{u}(u,v),\qquad\psi_m(u,v)=h_m(u,v)\pp{x_m}{v}(u,v).
\end{align}
Using \eqref{eq:258}, \eqref{eq:259}, \eqref{eq:262} (and the latter also with $h$ in the role of $g$) we obtain
\begin{align}
  \label{eq:266}
  |\Delta_m\phi|&\leq C\left(v^2\|\Delta_m\beta_{+}\|_1+v\left(\sup_{T_\varepsilon}\left|\Delta_m\ppp{x}{u}\right|+\sup_{T_\varepsilon}\left|\Delta_m\pppp{x}{u}{v}\right|\right)\right),\\
  \label{eq:267}
  |\Delta_m\psi|&\leq C\left(v^2\|\Delta_m\beta_{+}\|_1+v\left(\sup_{T_\varepsilon}\left|\Delta_m\ppp{x}{v}\right|+\sup_{T_\varepsilon}\left|\Delta_m\pppp{x}{u}{v}\right|\right)\right).
\end{align}
Using in addition also \eqref{eq:263} (and this one also with $h$ in the role of $g$) we obtain for differences of derivatives of $\phi$
\begin{align}
  \label{eq:268}
  \left|\Delta_m\pp{\phi}{u}\right|&\leq C\left(v\|\Delta_m\beta_{+}\|_1+\sup_{T_\varepsilon}\left|\Delta_m\ppp{x}{u}\right|+v\sup_{T_\varepsilon}\left|\Delta_m\pppp{x}{u}{v}\right|\right),\\
  \label{eq:269}
  \left|\Delta_m\pp{\phi}{v}\right|&\leq C\left(v\|\Delta_m\beta_{+}\|_1+v\sup_{T_\varepsilon}\left|\Delta_m\ppp{x}{u}\right|+\sup_{T_\varepsilon}\left|\Delta_m\pppp{x}{u}{v}\right|\right)
\end{align}
and for differences of derivatives of $\psi$
\begin{align}
  \label{eq:270}
  \left|\Delta_m\pp{\psi}{u}\right|&\leq C\left(v\|\Delta_m\beta_{+}\|_1+v\sup_{T_\varepsilon}\left|\Delta_m\ppp{x}{v}\right|+\sup_{T_\varepsilon}\left|\Delta_m\pppp{x}{u}{v}\right|\right),\\
  \label{eq:271}
    \left|\Delta_m\pp{\psi}{v}\right|&\leq C\left(v\|\Delta_m\beta_{+}\|_1+\sup_{T_\varepsilon}\left|\Delta_m\ppp{x}{v}\right|+v\sup_{T_\varepsilon}\left|\Delta_m\pppp{x}{u}{v}\right|\right).
\end{align}

We turn to $\Delta_mt(u,v)$. We recall
\begin{align}
  \label{eq:272}
  t_m(u,v)=\int_0^u(\phi_m+\psi_m)(u',u')du'+\int_u^v\psi_m(u,v')dv'.
\end{align}
Using \eqref{eq:266}, \eqref{eq:267} we obtain
\begin{align}
  \label{eq:273}
  |\Delta_mt|\leq C\left(v^3\|\Delta_m\beta_{+}\|_1+v^2\left(\sup_{T_\varepsilon}\left|\Delta_m\ppp{x}{u}\right|+\sup_{T_\varepsilon}\left|\Delta_m\ppp{x}{v}\right|+\sup_{T_\varepsilon}\left|\Delta_m\pppp{x}{u}{v}\right|\right)\right).
\end{align}
From $\pp{t_m}{v}(u,v)=\psi_m(u,v)$ we obtain, using \eqref{eq:267},
\begin{align}
  \label{eq:274}
  \left|\Delta_m\pp{t}{v}\right|\leq C\left(v^2\|\Delta_m\beta_{+}\|_1+v\left(\sup_{T_\varepsilon}\left|\Delta_m\ppp{x}{v}\right|+\sup_{T_\varepsilon}\left|\Delta_m\pppp{x}{u}{v}\right|\right)\right).
\end{align}
From
\begin{align}
  \label{eq:275}
  \pp{t_m}{u}(u,v)=\phi_m(u,u)+\int_u^v\pp{\psi_m}{u}(u,v')dv'
\end{align}
we obtain, using \eqref{eq:266}, \eqref{eq:270},
\begin{align}
  \label{eq:276}
  \left|\Delta_m\pp{t}{u}\right|\leq C\left(v^2\|\Delta_m\beta_{+}\|_1+v\left(\sup_{T_\varepsilon}\left|\Delta_m\ppp{x}{u}\right|+v\sup_{T_\varepsilon}\left|\Delta_m\ppp{x}{v}\right|+\sup_{T_\varepsilon}\left|\Delta_m\pppp{x}{u}{v}\right|\right)\right)
\end{align}
From $\pppp{t_m}{u}{v}(u,v)=\pp{\psi_m}{u}(u,v)$ we obtain, using \eqref{eq:270},
\begin{align}
  \label{eq:277}
  \left|\Delta_m\pppp{t}{u}{v}\right|\leq C\left(v\|\Delta_m\beta_{+}\|_1+v\sup_{T_\varepsilon}\left|\Delta_m\ppp{x}{v}\right|+\sup_{T_\varepsilon}\left|\Delta_m\pppp{x}{u}{v}\right|\right).
\end{align}
From $\ppp{t_m}{v}(u,v)=\pp{\psi_m}{v}(u,v)$ we obtain, using \eqref{eq:271}
\begin{align}
  \label{eq:278}
  \left|\Delta_m\ppp{t}{v}\right|\leq C\left(v\|\Delta_m\beta_{+}\|_1+\sup_{T_\varepsilon}\left|\Delta_m\ppp{x}{v}\right|+v\sup_{T_\varepsilon}\left|\Delta_m\pppp{x}{u}{v}\right|\right).
\end{align}
For $\Delta_m\ppp{t}{u}$ we start with the equation (see \eqref{eq:172})
\begin{align}
  \label{eq:279}
  \ppp{t_m}{u}(u,v)=\pp{\phi_m}{u}(u,u)+\pp{\phi_m}{v}(u,u)+\pp{}{u}\left(\int_u^v\pp{\psi_m}{u}(u,v')dv'\right).
\end{align}
Differences of the first two terms can be estimated by \eqref{eq:268}, \eqref{eq:269}. For the last term we use \eqref{eq:174}. Taking the difference of \eqref{eq:174} and itself but with $\psi_{m-1}$ in the role of $\psi_m$ and using the estimates for the differences $\Delta_mh$ and derivatives thereof as given in \eqref{eq:262}, \eqref{eq:263}, \eqref{eq:264} (but with $h$ in the role of $g$) together with the estimates for the differences $\Delta_mx$ and derivatives thereof as given in \eqref{eq:258}, \eqref{eq:259}, we obtain
\begin{align}
  \label{eq:280}
  \left|\Delta_m\left(\pp{}{u}\left(\int_u^v\pp{\psi}{u}(u,v')dv'\right)\right)\right|&\leq C\Bigg(v\|\Delta_m\beta_{+}\|_1+\sup_{T_\varepsilon}\left|\Delta_m\ppp{x}{u}\right|\nonumber\\
  &\qquad\qquad\qquad+v\sup_{T_\varepsilon}\left|\Delta_m\ppp{x}{v}\right|+\sup_{T_\varepsilon}\left|\Delta_m\pppp{x}{u}{v}\right|\Bigg).
\end{align}
Using this together with \eqref{eq:268}, \eqref{eq:269} to estimate differences of the first two terms in \eqref{eq:279} we obtain
\begin{align}
  \label{eq:281}
  \left|\Delta_m\ppp{t}{u}\right|\leq C\left(v\|\Delta_m\beta_{+}\|_1+\sup_{T_\varepsilon}\left|\Delta_m\ppp{x}{u}\right|+v\sup_{T_\varepsilon}\left|\Delta_m\ppp{x}{v}\right|+\sup_{T_\varepsilon}\left|\Delta_m\pppp{x}{u}{v}\right|\right).
\end{align}
Using the definition of $\|f\|_2$ as given in \eqref{eq:121} we have
\begin{align}
  \label{eq:282}
  |\Delta_mt|&\leq C\Big(v^3\|\Delta_m\beta_{+}\|_1+v^2\|\Delta_mx\|_2\Big),\\
  \label{eq:283}
  \left|\Delta_m\pp{t}{u}\right|,\left|\Delta_m\pp{t}{v}\right|&\leq C\Big(v^2\|\Delta_m\beta_{+}\|_1+v\|\Delta_mx\|_2\Big),\\
  \label{eq:284}
  \left|\Delta_m\pppp{t}{u}{v}\right|,\left|\Delta_m\ppp{t}{v}\right|,\left|\Delta_m\ppp{t}{u}\right|&\leq C\Big(v\|\Delta_m\beta_{+}\|_1+\|\Delta_mx\|_2\Big).
\end{align}
Now we look at $t_{+m}(v)=t_m(av,v)$, $x_{+m}(v)=x_m(av,v)$. From \eqref{eq:282}, \eqref{eq:283}, \eqref{eq:284} we have
\begin{align}
  \label{eq:285}
  |\Delta_mt_+|&\leq C\Big(v^3\|\Delta_m\beta_{+}\|_1+v^2\|\Delta_mx\|_2\Big),\\
  \label{eq:286}
  |\Delta_mt'_+|&\leq C\Big(v^2\|\Delta_m\beta_{+}\|_1+v\|\Delta_mx\|_2\Big),\\
  \label{eq:287}
  |\Delta_mt''_+|&\leq C\Big(v\|\Delta_m\beta_{+}\|_1+\|\Delta_mx\|_2\Big).
\end{align}
From \eqref{eq:258}, \eqref{eq:259}, \eqref{eq:260} we have
\begin{align}
  \label{eq:288}
  |\Delta_mx_+|&\leq v^2\|\Delta_mx\|_2,\\
  \label{eq:289}
  |\Delta_mx_+'|&\leq Cv\|\Delta_mx\|_2,\\
  \label{eq:290}
  |\Delta_mx_+''|&\leq C\|\Delta_mx\|_2.
\end{align}

We turn to estimates for $\Delta_m\alpha_-(v)$, $\Delta_m\beta_-(v)$ and derivatives thereof. In view of
\begin{align}
  \label{eq:291}
  \alpha_{-m}(v)=\alpha^\ast(t_{+m}(v),x_{+m}(v)),\qquad \beta_{-m}(v)=\beta^\ast(t_{+m}(v),x_{+m}(v)),
\end{align}
the above estimates for $\Delta_mt_+$, $\Delta_mx_+$ yield
\begin{align}
  \label{eq:292}
  |\Delta_m\alpha_-|,|\Delta_m\beta_-|&\leq C\Big(v^3\|\Delta_m\beta_{+}\|_1+v^2\|\Delta_mx\|_2\Big),\\
  \label{eq:293}
  |\Delta_m\alpha'_-|,|\Delta_m\beta'_-|&\leq C\Big(v^2\|\Delta_m\beta_{+}\|_1+v\|\Delta_mx\|_2\Big),\\
  \label{eq:294}
  |\Delta_m\alpha''_-|,|\Delta_m\beta''_-|&\leq C\Big(v\|\Delta_m\beta_{+}\|_1+\|\Delta_mx\|_2\Big).
\end{align}

We look at $\Delta_mV(v)$ and its derivative. We recall
 \begin{align}
   \label{eq:295}
   V_m(v)=\frac{\jump{\rho_m(v)w_m(v)}}{\jump{\rho_m(v)}},
 \end{align}
where
\begin{align}
  \label{eq:296}
  \jump{f_m(v)}=f_{+m}(v)-f_{-m}(v)
\end{align}
and
\begin{align}
  \label{eq:297}
  \rho_{\pm m}(v)=\rho(\alpha_{\pm m}(v),\beta_{\pm m}(v)),\qquad   w_{\pm m}(v)=w(\alpha_{\pm m}(v),\beta_{\pm m}(v)).
\end{align}
We recall that $\rho(\alpha,\beta)$, $w(\alpha,\beta)$ are smooth functions of their arguments. In view of \eqref{eq:200} and the analogous with $\beta$ in the role of $\alpha$, the points $(\alpha_{-m}(v),\beta_{-m}(v))$ and $(\alpha_{-m-1}(v),\beta_{-m-1}(v))$ lie in a ball in $\mathbb{R}^2$ centered at $(\alpha_0^\ast,\beta_0^\ast)$. Hence so does the line segment joining them (For the analogous statement involving $\alpha_+$, $\beta_+$, see right above \eqref{eq:262}). Therefore, the above estimates for $\Delta_m\alpha_\pm$, $\Delta_m\beta_\pm$ yield
\begin{align}
  \label{eq:298}
  |\Delta_mV|&\leq Cv^2\Big(\|\Delta_m\beta_{+}\|_1+\|\Delta_mx\|_2\Big),\\
  \label{eq:299}
  |\Delta_mV'|&\leq Cv\Big(\|\Delta_m\beta_{+}\|_1+\|\Delta_mx\|_2\Big).
\end{align}

We turn to estimate $\Delta_m\Gamma$ and its derivative. We recall
\begin{align}
  \label{eq:300}
  \Gamma_m(v)=a\frac{\cout_{+m}(v)}{\cin_{+m}(v)}\,\frac{V_m(v)-\cin_{+m}(v)}{\cout_{+m}(v)-V_m(v)},
\end{align}
where
\begin{align}
  \label{eq:301}
  \cin_{+m}(v)=\cin_m(av,v),\qquad \cout_{+m}(v)=\cout_m(av,v)
\end{align}
and
\begin{align}
  \label{eq:302}
  \cin_m(u,v)=\cin(\alpha_m(u,v),\beta_m(u,v)),\qquad\cout_m(u,v)=\cout(\alpha_m(u,v),\beta_m(u,v)).
\end{align}
We need to establish estimates for $\Delta_m\cin$, $\Delta_m\cout$ and their derivatives. $\cin(\alpha,\beta)$, $\cout(\alpha,\beta)$ are smooth functions of their arguments, so they satisfy the same estimates as $g_m(u,v)$, $h_m(u,v)$ (see \eqref{eq:262}, \eqref{eq:263}, \eqref{eq:264}). I.e.
\begin{align}
  \label{eq:303}
  |\Delta_m\cin|&\leq Cv^2\|\Delta_m\beta_{+}\|_1,\\
  \label{eq:304}
  \left|\Delta_m\pp{\cin}{u}\right|,\left|\Delta_m\pp{\cin}{v}\right|&\leq Cv\|\Delta_m\beta_{+}\|_1,\\
  \label{eq:305}
  \left|\Delta_m\ppp{\cin}{u}\right|,\left|\Delta_m\ppp{\cin}{v}\right|&\leq C\|\Delta_m\beta_{+}\|_1,\qquad \left|\Delta_m\pppp{\cin}{u}{v}\right|\leq Cv\|\Delta_m\beta_{+}\|_1
\end{align}
and the same estimates hold with $\cout$ in the role of $\cin$. Using  \eqref{eq:298}, \eqref{eq:299} together with \eqref{eq:303}, \eqref{eq:304} we obtain
\begin{align}
\label{eq:306}
  |\Delta_m\Gamma|\leq&Cv^2\Big(\|\Delta_m\beta_{+}\|_1+\|\Delta_mx\|_2\Big),\\
\label{eq:307}
  |\Delta_m\Gamma'|\leq&Cv\Big(\|\Delta_m\beta_{+}\|_1+\|\Delta_mx\|_2\Big).
\end{align}

We turn to estimating $\Delta_mM$ and derivatives thereof. In order to do this we first need to estimate differences of $\mu$, $\nu$ and derivatives thereof. We recall
\begin{align}
  \label{eq:308}
  \mu_m(u,v)&=\frac{1}{\cout_m(u,v)-\cin_m(u,v)}\,\frac{\cout_m(u,v)}{\cin_m(u,v)}\,\pp{\cin_m}{v}(u,v),\\
  \label{eq:309}
  \nu_m(u,v)&=\frac{1}{\cout_m(u,v)-\cin_m(u,v)}\,\frac{\cin_m(u,v)}{\cout_m(u,v)}\,\pp{\cout_m}{u}(u,v).
\end{align}
Using \eqref{eq:303}, \eqref{eq:304}, \eqref{eq:305} we obtain
\begin{align}
  \label{eq:310}
  |\Delta_m\mu|,|\Delta_m\nu|&\leq Cv\|\Delta_m\beta_{+}\|_1,\\
  \label{eq:311}
  \left|\Delta_m\pp{\mu}{u}\right|,\left|\Delta_m\pp{\mu}{v}\right|,\left|\Delta_m\pp{\nu}{u}\right|,\left|\Delta_m\pp{\nu}{v}\right|&\leq C\|\Delta_m\beta_{+}\|_1.
\end{align}
In view of
\begin{align}
  \label{eq:312}
  M_m(u,v)=\mu_m(u,v)\pp{x_m}{u}(u,v)+\nu_m(u,v)\pp{x_m}{v}(u,v),
\end{align}
the estimates \eqref{eq:310}, \eqref{eq:311} together with \eqref{eq:258}, \eqref{eq:259} imply
\begin{align}
  \label{eq:313}
  |\Delta_mM|&\leq Cv\Big(\|\Delta_m\beta_{+}\|_1+\|\Delta_mx\|_2\Big),\\
\label{eq:314}
  \left|\Delta_m\pp{M}{u}\right|,\left|\Delta_m\pp{M}{v}\right|&\leq C\Big(\|\Delta_m\beta_{+}\|_1+\|\Delta_mx\|_2\Big).
\end{align}

Now we estimate differences of the second derivatives of $x_{m+1}(u,v)$. From (see \eqref{eq:226})
\begin{align}
  \label{eq:315}
  \pppp{x_{m+1}}{u}{v}(u,v)=M_m(u,v)
\end{align}
we obtain
\begin{align}
  \label{eq:316}
  \left|\Delta_{m+1}\pppp{x}{u}{v}\right|\leq Cv\Big(\|\Delta_m\beta_{+}\|_1+\|\Delta_mx\|_2\Big).
\end{align}
To estimate $\Delta_{m+1}\ppp{x}{u}(u,v)$ we use \eqref{eq:229}:
\begin{align}
  \label{eq:317}
    \ppp{x_{m+1}}{u}(u,v)&=-\Gamma_{m}(u)a\ppp{x_m}{u}(au,u)-\Gamma_{m}'(u)\pp{x_{m}}{u}(au,u)-\Gamma_{m}(u)M_m(au,u)\nonumber\\
                         &\qquad +aM_m(au,u) -2M_m(u,u)\nonumber\\
                         &\qquad-\int_{au}^u\pp{M_m}{v}(u',u)du'+\int_u^v\pp{M_m}{u}(u,v')dv'.
\end{align}
Forming the difference from the first term we get
\begin{align}
  \label{eq:318}
  \left|\Delta_m\left(-\Gamma(u)a\ppp{x}{u}(au,u)\right)\right|&\leq |\Gamma_m(u)a|\left|\Delta_{m}\ppp{x}{u}(au,u)\right|+\left|\ppp{x_{m-1}}{u}(au,u)\right||a\Delta_m\Gamma(u)|\nonumber\\
                                                  &\leq (a+Cu)\left|\Delta_{m}\ppp{x}{u}(au,u)\right|+C|\Delta_m\Gamma(u)|,
\end{align}
where we used \eqref{eq:213}. Using this and also \eqref{eq:258}, \eqref{eq:306}, \eqref{eq:307}, \eqref{eq:313}, \eqref{eq:314} we obtain
\begin{align}
  \label{eq:319}
  \left|\Delta_{m+1}\ppp{x}{u}\right|\leq a\sup_{T_\varepsilon}\left|\Delta_m\ppp{x}{u}\right|+Cv\Big(\|\Delta_m\beta_{+}\|_1+\|\Delta_mx\|_2\Big).
\end{align}
To estimate $\Delta_{m+1}\ppp{x}{v}(u,v)$ we use \eqref{eq:230}:
\begin{align}
  \label{eq:320}
  \ppp{x_{m+1}}{v}(u,v)&=\Gamma_{m}(v)a\ppp{x_{m}}{u}(av,v)+\Gamma_{m}'(v)\pp{x_{m}}{u}(av,v)+\Gamma_{m}(v)M_m(av,v)\nonumber\\
                       &\qquad -aM_m(av,v)+\int_{av}^u\pp{M_m}{v}(u',v)du',
\end{align}
Analogously to the way we arrived at \eqref{eq:319} we obtain
\begin{align}
  \label{eq:321}
  \left|\Delta_{m+1}\ppp{x}{v}\right|\leq a\sup_{T_\varepsilon}\left|\Delta_m\ppp{x}{u}\right|+Cv\Big(\|\Delta_m\beta_{+}\|_1+\|\Delta_mx\|_2\Big).
\end{align}
Choosing $\varepsilon$ sufficiently small, we can rewrite the estimates \eqref{eq:316}, \eqref{eq:319}, \eqref{eq:321} as
\begin{align}
  \label{eq:322}
  \left|\Delta_{m+1}\pppp{x}{u}{v}\right|,\left|\Delta_{m+1}\ppp{x}{u}\right|,\left|\Delta_{m+1}\ppp{x}{v}\right|&\leq Cv\|\Delta_m\beta_{+}\|_1+k\|\Delta_mx\|_2,
\end{align}
where $k$ is a constant satisfying
\begin{align}
  \label{eq:323}
  0<a<k<1.
\end{align}
I.e.~we have
\begin{align}
  \label{eq:324}
  \|\Delta_{m+1}x\|_2\leq C\varepsilon\|\Delta_m\beta_{+}\|_1+k\|\Delta_mx\|_2.
\end{align}

We turn to estimate $\Delta_{m+1}\beta_+''$. We use \eqref{eq:245}
\begin{align}
  \label{eq:325}
  \beta_{+,m+1}''(v)&=F(\alpha_{+m}(v),\alpha_{-m}(v),\beta_{-m}(v))\alpha_{+m}''(v)\nonumber\\
                    &\qquad +M_1(\alpha_{+m}(v),\alpha_{-m}(v),\beta_{-m}(v))\beta_{-m}''(v)\nonumber\\
                    &\qquad +M_2(\alpha_{+m}(v),\alpha_{-m}(v),\beta_{-m}(v))\alpha_{-m}''(v)\nonumber\\
                    &\qquad +G(\alpha_{+m}(v),\alpha_{-m}(v),\beta_{-m}(v),\alpha_{+m}'(v),\alpha_{-m}'(v),\beta_{-m}'(v))
\end{align}
and recall that $F$, $M_1$, $M_1$, $G$ are smooth functions of their arguments. Furthermore we recall, (see \eqref{eq:61}),
\begin{align}
\label{eq:326}
  F(\alpha_{+m}(0),\alpha_{-m}(0),\beta_{-m}(0))=-a^2.
\end{align}
From \eqref{eq:325} we deduce
\begin{align}
  \label{eq:327}
  |\Delta_{m+1}\beta_+''|&\leq (a^2+Cv)|\Delta_m\alpha_+''|\nonumber\\
                         &\qquad+C\Big(|\Delta_m\alpha_+|+|\Delta_m\alpha_-|+|\Delta_m\beta_-|+|\Delta_m\alpha_+'|+|\Delta_m\alpha_-'|+|\Delta_m\beta_-'|\nonumber\\
  &\hspace{80mm}+|\Delta_m\alpha_-''|+|\Delta_m\beta_-''|\Big).
\end{align}
Using \eqref{eq:255}, \eqref{eq:292}, \eqref{eq:293}, \eqref{eq:294} we obtain
\begin{align}
  \label{eq:328}
  |\Delta_{m+1}\beta_+''|&\leq (a^4+Cv)\|\Delta_m\beta_{+}\|_1+C\|\Delta_mx\|_2.
\end{align}
Choosing $\varepsilon$ sufficiently small and taking the supremum in $[0,\varepsilon]$, we have (see \eqref{eq:323} for the constant $k$)
\begin{align}
  \label{eq:329}
  \|\Delta_{m+1}\beta_{+}\|_1&\leq k\|\Delta_m\beta_{+}\|_1+C\|\Delta_mx\|_2.
\end{align}

In obvious notation, the estimates \eqref{eq:324}, \eqref{eq:329} can be written as
\begin{align}
  \label{eq:330}
  \begin{pmatrix}
    \|\Delta_{m+1}\beta_{+}\|_1 \\ \|\Delta_{m+1}x\|_2
  \end{pmatrix}\leq
A                              \begin{pmatrix}
                                \|\Delta_{m}\beta_{+}\|_1 \\ \|\Delta_{m}x\|_2
                              \end{pmatrix},
\end{align}
where $A$ is the matrix
\begin{align}
  \label{eq:331}
      A=\begin{pmatrix}
      k & C \\ C\varepsilon & k
    \end{pmatrix}.
\end{align}
The eigenvalues of $A$ are
\begin{align}
  \label{eq:332}
  \lambda_\pm=k\pm C\sqrt{\varepsilon}.
\end{align}
Choosing $\varepsilon$ sufficiently small we have $|\lambda_\pm|< 1$. It follows that the sequence
  \begin{align}
\label{eq:333}
    ((\beta_{+m},x_m);m=0,1,2,\ldots)   
  \end{align}
converges in $B_B\times B_{N_0}$.
\end{proof}

\subsection{Proof of the Existence Theorem}
The two propositions above show that the sequence $((\beta_{+m},x_m);m=0,1,2,\ldots)$ converges uniformly in $[0,\varepsilon]\times T_\varepsilon$ to $(\beta_+,x)\in B_B\times B_{N_0}$. From $\beta_m(u,v)=\beta_{+m}(v)$, $\alpha_m(u,v)=\beta_{+m}(u)$ it follows that $\alpha_m$, $\beta_m$ converge uniformly in $T_\varepsilon$ and the limits $\alpha$, $\beta$ satisfy
\begin{align}
  \label{eq:334}
  \pp{\alpha}{v}=0=\pp{\beta}{u}.
\end{align}
The convergence of $\alpha_m$, $\beta_m$ implies the convergence of $\cin_m$, $\cout_m$ to $\cin$, $\cout$, given by\footnote{We recall that in the proofs of the above propositions, the notation
  \begin{align}
    \label{eq:335}
    \cin_m(u,v)=\cin(\alpha_m(u,v),\beta_m(u,v))
  \end{align}
was used. The notation $\cin(u,v)$, $\cout(u,v)$ (without an index and the arguments are $u$, $v$) is introduced as the limit of $\cin_m(u,v)$, $\cout_m(u,v)$.}
\begin{align}
  \label{eq:336}
  \cin(u,v)=\cin(\alpha(u,v),\beta(u,v)),\qquad \cout(u,v)=\cout(\alpha(u,v),\beta(u,v)),
\end{align}
which in turn implies the convergence of $\mu_m$, $\nu_m$ to $\mu$, $\nu$, given by
\begin{align}
  \label{eq:337}
  \mu=\frac{1}{\cout-\cin}\frac{\cout}{\cin}\pp{\cin}{v},\qquad \nu=\frac{1}{\cout-\cin}\frac{\cin}{\cout}\pp{\cout}{u}.
\end{align}
This together with the convergence of $x_m$ to $x$ implies, in view of \eqref{eq:312}, the convergence of $M_m$ to $M$ which is given by
\begin{align}
  \label{eq:338}
  M=\mu\pp{x}{u}+\nu\pp{x}{v}.
\end{align}
From \eqref{eq:315} we see that the limits $x$ and $M$ satisfy
\begin{align}
  \label{eq:339}
  \pppp{x}{u}{v}(u,v)=M(u,v).
\end{align}

The convergence of $\alpha_m$, $\beta_m$ to $\alpha$, $\beta$ implies the convergence of $g_m$, $h_m$ to $g$, $h$ (see \eqref{eq:152}), given by
\begin{align}
  \label{eq:340}
  g(u,v)=\frac{1}{\cin(u,v)},\qquad h(u,v)=\frac{1}{\cout(u,v)},
\end{align}
which, together with the convergence of $x_m$ to $x$, in view of \eqref{eq:265}, implies the convergence of $\phi_m$, $\psi_m$ to $\phi$, $\psi$ given by
\begin{align}
  \label{eq:341}
  \phi(u,v)&=\frac{1}{\cin(u,v)}\pp{x}{u}(u,v),\\
  \label{eq:342}
  \psi(u,v)&=\frac{1}{\cout(u,v)}\pp{x}{v}(u,v).
\end{align}

From \eqref{eq:149} we have that $t_m$ converges to $t$ given by
\begin{align}
  \label{eq:343}
  t(u,v)=\int_0^u\left(\phi+\psi\right)(u',u')du'+\int_u^v\psi(u,v')dv'.
\end{align}
Hence,
\begin{align}
  \label{eq:344}
  \pp{t}{v}(u,v)=\psi(u,v)=\frac{1}{\cout(u,v)}\pp{x}{v}(u,v)
\end{align}
and 
\begin{align}
  \label{eq:345}
  \pp{t}{u}(u,v)=\phi(u,u)+\int_u^v\pp{\psi}{u}(u,v')dv'.
\end{align}
It follows by a direct computation that \eqref{eq:339} (using also \eqref{eq:337}, \eqref{eq:338}) is equivalent to
\begin{align}
  \label{eq:346}
  \pp{\psi}{u}(u,v)=\pp{\phi}{v}(u,v).
\end{align}
Using this in \eqref{eq:345} we obtain
\begin{align}
  \label{eq:347}
  \pp{t}{u}(u,v)&=\phi(u,u)+\int_u^v\pp{\phi}{v}(u,v')dv'\nonumber\\
                &=\phi(u,u)+\phi(u,v)-\phi(u,u)\nonumber\\
                &=\phi(u,v)\nonumber\\
                &=\frac{1}{\cin(u,v)}\pp{x}{u}(u,v).
\end{align}
Equations \eqref{eq:344}, \eqref{eq:347} are the characteristic equations, i.e.~the functions $x(u,v)$, $t(u,v)$ satisfy the characteristic equations in $T_\varepsilon$.

From $x_m(u,u)=0$, $\alpha_m(u,u)=\beta_m(u,u)$ we have
\begin{align}
  \label{eq:348}
  x(u,u)=0,\qquad \alpha(u,u)=\beta(u,u),
\end{align}
i.e.~the boundary conditions on the wall are satisfied in the limit.

The convergence of $\alpha_m$, $\beta_m$ to $\alpha$, $\beta$ implies the convergence of the boundary functions $\alpha_{+m}$, $\beta_{+m}$ to $\alpha_+$, $\beta_+$, given by
\begin{align}
  \label{eq:349}
  \alpha_+(v)=\alpha(av,v),\qquad \beta_+(v)=\beta_+(av,v).
\end{align}
The convergence of $t_m$, $x_m$ to $t$, $x$ implies the convergence of the boundary functions $t_{+m}$, $x_{+m}$ to $t_+$, $x_+$, given by
\begin{align}
  \label{eq:350}
  x_+(v)=x(av,v),\qquad t_+=t(av,v).
\end{align}
This in turn implies the convergence of $\alpha_{-m}$, $\beta_{-m}$ to $\alpha_-$, $\beta_-$ (see \eqref{eq:195}), given by
\begin{align}
  \label{eq:351}
  \alpha_-(v)=\alpha^\ast(t_+(v),x_+(v)),\qquad\beta_-(v)=\beta^\ast(t_+(v),x_+(v)).
\end{align}
The convergence of the boundary functions then implies the convergence of $V_m$ to $V$, given by
\begin{align}
  \label{eq:352}
  V(v)=\frac{\jump{\rho(v)w(v)}}{\jump{\rho(v)}},
\end{align}
where
\begin{align}
  \label{eq:353}
  \rho_\pm(v)=\rho(\alpha_\pm(v),\beta_\pm(v)),\qquad w_\pm(v)=w(\alpha_\pm(v),\beta_\pm(v)).
\end{align}
This implies the convergence of $\Gamma_m$ to $\Gamma$, given by
\begin{align}
  \label{eq:354}
  \Gamma(v)=a\frac{\cout_+(v)}{\cin_+(v)}\,\frac{V(v)-\cin_+(v)}{\cout_+(v)-V(v)}.
\end{align}
The convergence of $x_m$, $\Gamma_m$ and $M_m$ imply the convergence of $\Lambda_m$ to $\Lambda$ (see \eqref{eq:222}), given by
\begin{align}
  \label{eq:355}
  \Lambda(u)&=\Gamma(u)a\ppp{x}{u}(au,u)+\Gamma'(u)\pp{x}{u}(au,u)+\Gamma(u)M(au,u)\nonumber\\
            &=\Gamma(u)a\ppp{x}{u}(au,u)+\Gamma'(u)\pp{x}{u}(au,u)+\Gamma(u)\pppp{x}{u}{v}(au,u)\nonumber\\
            &=\pp{}{u}\left(\Gamma(u)\pp{x}{u}(au,u)\right).
\end{align}
The limiting equation of \eqref{eq:224} with $u=av$ is
\begin{align}
  \label{eq:356}
  \pp{x}{v}(av,v)&=1+\Gamma(v)\pp{x}{u}(av,v)-\Gamma(0)\pp{x}{u}(0,0)\nonumber\\
                 &=\Gamma(v)\pp{x}{u}(av,v),
\end{align}
where for the second equality we used $\Gamma(0)=-1$, $\pp{x}{u}(0,0)=-1$. Using the characteristic equations \eqref{eq:344}, \eqref{eq:347}, the boundary condition \eqref{eq:356} is equivalent to (see computations between \eqref{eq:27} and \eqref{eq:37})
\begin{align}
  \label{eq:357}
  \frac{dx_+}{dv}(v)=V(v)\frac{dt_+}{dv}(v),
\end{align}
i.e.~the boundary condition on the shock is satisfied in the limit.

Equation \eqref{eq:239} is satisfied in the limit, i.e.
\begin{align}
  \label{eq:358}
  \beta_+(v)=H\Big(\alpha_+(v),\alpha_-(v),\beta_-(v)\Big).
\end{align}
This implies (see \eqref{eq:55})
\begin{align}
  \label{eq:359}
  J(\alpha_+,\beta_+,\alpha_-,\beta_-)=0,
\end{align}
i.e.~the other jump condition along the shock is also satisfied in the limit.

We recall that the determinism condition is satisfied at the reflection point by assumption on the data (see \eqref{eq:15}):
\begin{align}
  \label{eq:360}
    (\cout_0^\ast)_0<V_0<\eta_0,
\end{align}
where
\begin{align}
  \label{eq:361}
  (\cout_0^\ast)_0&=\cout_-(0)=\cout(\alpha_-(0),\beta_-(0)),\\
  \label{eq:362}
  \eta_0&=\cout_+(0)=\cout(\alpha_+(0),\beta_+(0)).
\end{align}
Therefore, choosing $\varepsilon$ sufficiently small, the determinsim condition is satisfied for $v\in[0,\varepsilon]$, i.e.
\begin{align}
  \label{eq:363}
  \cout_-(v)<V(v)<\cout_+(v).
\end{align}

In view of \eqref{eq:334}, \eqref{eq:344}, \eqref{eq:347}, \eqref{eq:352}, \eqref{eq:359}, \eqref{eq:363} we have proven the existence of a twice differentiable solution $(\alpha,\beta,t,x)$ to the shock reflection problem as stated in subsection \ref{shock_reflection_problem}.

The determinant of the Jacobian $\frac{\partial(t,x)}{\partial(u,v)}$ is given by
\begin{align}
  \label{eq:364}
  \begin{vmatrix}
    \pp{t}{u} & \pp{t}{v}\\
    \pp{x}{u} & \pp{x}{v}
  \end{vmatrix}=(\cout-\cin)\pp{t}{u}\pp{t}{v}=2\eta\pp{t}{u}\pp{t}{v}=\frac{2}{\eta_0}+\Landau(v).
\end{align}
Therefore, by choosing $\varepsilon$ sufficiently small, the functions $\alpha$, $\beta$ are twice differentiable functions of $(t,x)$ on the image of $T_\varepsilon$ by the map $(u,v)\mapsto(t(u,v),x(u,v))$.

This concludes the proof of theorem \ref{existence_theorem} from page \pageref{existence_theorem}.

\subsection{Asymptotic Form}
By straightforward expansions of the functions $(\alpha,\beta,t,x)$ at the reflection point in the state behind, one can show that any twice differentiable solution of the reflection problem is of the same asymptotic form as the solution constructed in the existence proof:
\begin{align*}
  \alpha(u,v)&=\beta_0+\beta_0'u+\Landau(v^2),\\
  \beta(u,v)&=\beta_0+\beta_0'v+\Landau(v^2),\\
  t(u,v)&=\frac{1}{\eta_0}(u+v)+\Landau(v^2),\\
  x(u,v)&=v-u+\Landau(v^2).
\end{align*}

\subsection{Uniqueness}
We have the following uniquenss result:
\begin{theorem}[Uniqueness]
  Let $(\alpha_1,\beta_1,t_1,x_1)$, $(\alpha_2,\beta_2,t_2,x_2)$, both in $C^2(T_\varepsilon)$, be two solutions of the reflection problem as stated in \ref{shock_reflection_problem} corresponding to the same future development of the data. Then, for $\varepsilon$ sufficiently small, the two solutions coincide.
\end{theorem}
Using similar estimates as in the convergence proof above, the proof is straightforward.

\section{Higher Regularity}
\begin{theorem}[Higher Regularity]
  For $\varepsilon$ sufficiently small, the established solution $\alpha$, $\beta$, $t$, $x$ of the reflection problem is infinitely differentiable.
\end{theorem}

\begin{proof}
We show that all derivatives of the functions $\alpha$, $\beta$, $t$, $x$ are bounded. This can be done by induction, i.e.~once it is assumed that the $n$'th order derivatives are bounded it can be shown that the $n+1$'th order derivatives are bounded. The base case of this induction is given by the solution being in $C^2(T_{\varepsilon})$, i.e.~it is already shown that the solution is two times continuously differentiable. Instead of showing the inductive step $n\mapsto n+1$ we only show that the third order derivatives are bounded, the general inductive step being completely analogous. We restrict to this case in order to simplify and shorten the presentation of the argument. However, it is important to note that the encountered smallness conditions on $\varepsilon$ in the process are independent of the order of derivatives studied, i.e.~also for the general inductive step, no other smallness conditions would be necessary. This will become apparent during the argument to follow.

Through the characteristic equations \eqref{eq:19}, the derivatives of $t(u,v)$ can be expressed in terms of the derivatives of $x(u,v)$ and functions of $\alpha$, $\beta$. Therefore, the bounds on the derivatives of $t(u,v)$ will follow from the bounds on the derivatives of $x(u,v)$. From this together with
\begin{align}
  \label{eq:365}
  \beta(u,v)=\beta_+(v),\qquad \alpha(u,v)=\beta_+(u),  
\end{align}
we see that it suffices to establish bounds on the derivatives of the functions $x(u,v)$, $\beta_+(v)$.

We first consider the function $\beta_+(v)$. We recall \eqref{eq:56}:
  \begin{align}
    \label{eq:366}
    \beta_+(v)=H(\alpha_+(v),\alpha_-(v),\beta_-(v)).    
  \end{align}
Here $H$ is a smooth function of its arguments and
  \begin{align}
    \label{eq:367}
  \alpha_+(v)=\alpha(av,v),\qquad \alpha_-(v)=\alpha^\ast(t_+(v),x_+(v)),\qquad\beta_-(v)=\beta^\ast(t_+(v),x_+(v)),    
  \end{align}
where
\begin{align}
      \label{eq:368}
  t_+(v)=t(av,v),\qquad x_+(v)=x(av,v).
\end{align}
Taking the third derivative of \eqref{eq:366} we obtain
\begin{align}
  \label{eq:369}
  \beta_+'''(v)&=F\Big(\alpha_+(v),\alpha_-(v),\beta_-(v)\Big)\alpha_+'''(v)\nonumber\\
             &\quad+M_1\Big(\alpha_+(v),\alpha_-(v),\beta_-(v)\Big)\beta_-'''(v)\nonumber\\
             &\quad+M_2\Big(\alpha_+(v),\alpha_-(v),\beta_-(v)\Big)\alpha_-'''(v)+\lot.
\end{align}
Here and in the following we denote by $\lot$ terms that are of lower order. Since we are studying the third order derivatives this means we denote by $\lot$ terms which can involve the functions $\alpha$, $\beta$, $t$, $x$, derivatives thereof and second derivatives thereof. Terms of lower order are bounded in absolute value. The functions $F$, $M_1$, $M_2$ correspond to the partial derivatives of $H$ and are therefore smooth functions of their arguments. They already show up in \eqref{eq:57}.

Taking the third derivative of $\alpha_-(v)=\alpha^\ast(t_+(v),x_+(v))$, we obtain
\begin{align}
  \label{eq:370}
  \alpha_-'''(v)=\pp{\alpha^\ast}{t}(t_+(v),x_+(v))t_+'''(v)+\pp{\alpha^\ast}{x}(t_+(v),x_+(v))x_+'''(v)+\lot.
\end{align}
The third derivative of $x_+(v)=x(av,v)$ is given by
\begin{align}
  \label{eq:371}
  x_+'''(v)=a^3\frac{\partial^3x}{\partial u^3}(av,v)+3a^2\frac{\partial^3x}{\partial u^2\partial v}(av,v)+3a\frac{\partial^3x}{\partial u\partial v^2}(av,v)+\frac{\partial^3x}{\partial v^3}(av,v).
\end{align}
The same equation holds with $t$ in the role of $x$. In view of \eqref{eq:80}, the partial derivatives of mixed type can be expressed by lower order derivatives, hence
\begin{align}
  \label{eq:372}
  x_+'''(v)=a^3\frac{\partial^3x}{\partial u^3}(av,v)+\frac{\partial^3x}{\partial v^3}(av,v)+\lot.
\end{align}
In view of the characteristic equations \eqref{eq:19} we have
\begin{align}
  \label{eq:373}
  \frac{\partial^3t}{\partial u^3}(u,v)&=\frac{1}{\cin(u,v)}\frac{\partial^3x}{\partial u^3}(u,v)+\lot,\\
  \label{eq:374}
  \frac{\partial^3t}{\partial v^3}(u,v)&=\frac{1}{\cout(u,v)}\frac{\partial^3x}{\partial u^3}(u,v)+\lot,
\end{align}
where we use the short notation $\cin(u,v)=\cin(\alpha(u,v),\beta(u,v))$, $\cout(u,v)=\cout(\alpha(u,v),\beta(u,v))$. The partial derivatives of $t(u,v)$ of mixed type can be expressed by partial derivatives of $x(u,v)$ of mixed type. Therefore,
\begin{align}
  \label{eq:375}
  t_+'''(v)=\frac{a^3}{\cin(av,v)}\frac{\partial^3x}{\partial u^3}(av,v)+\frac{1}{\cout(av,v)}\frac{\partial^3x}{\partial v^3}(av,v)+\lot.
\end{align}
Substituting \eqref{eq:372}, \eqref{eq:375} into \eqref{eq:370} we obtain
\begin{align}
  \label{eq:376}
  |\alpha_-'''(v)|,|\beta_-'''(v)|\leq C+\overline{C}_1\sup_{[0,\varepsilon]}\left|\frac{\partial^3x}{\partial u^3}(av,v)\right|+\overline{C}_2\sup_{[0,\varepsilon]}\left|\frac{\partial^3x}{\partial v^3}(av,v)\right|,
\end{align}
where the statement for $\beta_-'''(v)$ is obtained analogously. Here and in the following we denote by $C$ and $\overline{C}$ (possibly with an index) generic numerical constants whose values may change from line to line, in agreement with standard notation. However, by $\overline{C}$ we denote numerical constants whose values are the same no matter what order of derivative is studied. For example if we would estimate the fourth order derivatives (assuming third order derivatives are bounded) equation \eqref{eq:376} would look the same, just with fourth order derivatives instead of third order. But the constant $C$ would possess a different numerical value while the constants $\overline{C}_1$, $\overline{C}_2$ would possess the same numerical values as in \eqref{eq:376}.

For the function $F$ in \eqref{eq:369} we have the estimate (recall \eqref{eq:61})
\begin{align}
  \label{eq:377}
  \left|F\Big(\alpha_+(v),\alpha_-(v),\beta_-(v)\Big)\right|\leq a^2+\overline{C}v.
\end{align}
Using this together with \eqref{eq:376}  and $\alpha_+(v)=\beta_+(av)$ in \eqref{eq:369} we obtain
\begin{align}
  \label{eq:378}
  |\beta_+'''(v)|\leq C+(a^2+\overline{C}_1v)a^3|\beta_+'''(av)|+\overline{C}_2\sup_{[0,\varepsilon]}\left|\frac{\partial^3x}{\partial u^3}(av,v)\right|+\overline{C}_3\sup_{[0,\varepsilon]}\left|\frac{\partial^3x}{\partial v^3}(av,v)\right|,
\end{align}
which, by choosing $\varepsilon$ sufficiently small (recall that $a<1$), implies
\begin{align}
  \label{eq:379}
  \sup_{[0,\varepsilon]}|\beta_+'''(v)|\leq C+\overline{C}_1\sup_{[0,\varepsilon]}\left|\frac{\partial^3x}{\partial u^3}(av,v)\right|+\overline{C}_2\sup_{[0,\varepsilon]}\left|\frac{\partial^3x}{\partial v^3}(av,v)\right|.
\end{align}

We now consider the function $x(u,v)$. We recall \eqref{eq:77}:
\begin{align}
    \label{eq:380}
  \pp{x}{v}(u,v)=\Gamma(v)\pp{x}{u}(av,v)+\int_{av}^uM(u',v)du',
\end{align}
where (see \eqref{eq:36})
\begin{align}
    \label{eq:381}
    \Gamma(v)= a\frac{\cout_+(v)}{\cin_+(v)}\,\frac{V(v)-\cin_+(v)}{\cout_+(v)-V(v)}
\end{align}
and $V(v)$ satisfies
\begin{align}
  \label{eq:382}
  V(v)=\frac{\jump{\rho(v) w(v)}}{\jump{\rho(v)}},
\end{align}
where $\rho_+(v)=\rho(\alpha_+(v),\beta_+(v))$, $\rho_-(v)=\rho(\alpha_-(v),\beta_-(v))$ and analogous for $w_\pm(v)$. We see that $\Gamma(v)$ is a function of $\alpha$, $\beta$, $t$, $x$, hence,
\begin{align}
  \label{eq:383}
  \Gamma(v),\Gamma'(v), \Gamma''(v)=\lot.
\end{align}
We recall that
\begin{align}
  \label{eq:384}
    M=\pppp{x}{u}{v}=\mu\pp{x}{u}+\nu\pp{x}{v},
\end{align}
where
\begin{align}
  \label{eq:385}
    \mu= \frac{1}{\cout-\cin}\,\frac{\cout}{\cin}\pp{\cin}{v},\qquad \nu=-\frac{1}{\cout-\cin}\,\frac{\cin}{\cout}\pp{\cout}{u}.
\end{align}
We observe that $\mu$, $\nu$ involve first derivatives of $\beta_+(v)$. Therefore,
\begin{align}
  \label{eq:386}
  M(u,v),\pp{M}{v}(u,v)=\lot.
\end{align}
Using \eqref{eq:383}, \eqref{eq:386} and taking two derivatives with respect to $v$ of \eqref{eq:380} we obtain
\begin{align}
  \label{eq:387}
  \frac{\partial^3x}{\partial v^3}(u,v)=\Gamma(v)a^2\frac{\partial^3x}{\partial u^3}(av,v)+\int_{av}^u\ppp{M}{v}(u',v)du'+\lot.
\end{align}
Because of
\begin{align}
  \label{eq:388}
  \frac{\partial^3\cin}{\partial v^3}=\pp{\cin}{\beta}\frac{\partial^3\beta}{\partial v^3}+\lot,\qquad \frac{\partial^3\cout}{\partial u\partial v^2}=\lot,
\end{align}
we have
\begin{align}
  \label{eq:389}
  \ppp{\mu}{v}=\frac{1}{\cout-\cin}\frac{\cout}{\cin}\pp{\cin}{\beta}\frac{\partial^3\beta}{\partial v^3}+\lot,\qquad\ppp{\nu}{v}=\lot.
\end{align}
Therefore,
\begin{align}
  \label{eq:390}
  \ppp{M}{v}=\frac{1}{\cout-\cin}\frac{\cout}{\cin}\pp{\cin}{\beta}\frac{\partial^3\beta}{\partial v^3}\pp{x}{u}+\nu\frac{\partial^3x}{\partial v^3}+\lot,
\end{align}
which implies
\begin{align}
  \label{eq:391}
  \left|\ppp{M}{v}(u,v)\right|\leq C+\overline{C}_1|\beta_+'''(v)|+\overline{C}_2\left|\frac{\partial^3x}{\partial v^3}(u,v)\right|.
\end{align}
Putting this into \eqref{eq:387} and taking $\varepsilon$ sufficiently small, such that $|a\Gamma(v)|\leq 1$ (recall that $\Gamma(0)=-1$), we obtain
\begin{align}
  \label{eq:392}
  \left|\frac{\partial^3x}{\partial v^3}(u,v)\right|\leq C+a\left|\frac{\partial^3x}{\partial u^3}(av,v)\right|+\overline{C}_1u|\beta_+'''(v)|+\overline{C}_2\int_{av}^u\left|\frac{\partial^3x}{\partial v^3}(u',v)\right|du'.
\end{align}
Analogous to the way we arrived at this equation, we find, based on (see \eqref{eq:75})
\begin{align}
  \label{eq:393}
  \pp{x}{u}(u,v)=-\Gamma(u)\pp{x}{u}(au,u)-\int_{au}^uM(u',u)du'+\int_u^vM(u,v')dv'
\end{align}
the estimate
\begin{align}
  \label{eq:394}
  \left|\frac{\partial^3x}{\partial u^3}(u,v)\right|&\leq C+a\left|\frac{\partial^3x}{\partial u^3}(au,u)\right|+\overline{C}_1v|\beta_+'''(u)|+\overline{C}_2\int_{au}^u\left|\frac{\partial^3x}{\partial v^3}(u',u)\right|du'\nonumber\\
  &\hspace{60mm}+\overline{C}_3\int_u^v\left|\frac{\partial^3x}{\partial u^3}(u,v')\right|dv'.
\end{align}

Using a Gronwall type argument in the $u$-variable, we obtain from \eqref{eq:392}
\begin{align}
  \label{eq:395}
  \left|\frac{\partial^3x}{\partial v^3}(u,v)\right|\leq C+(a+\overline{C}_1u)\left|\frac{\partial^3x}{\partial u^3}(av,v)\right|+\overline{C}_2u|\beta_+'''(v)|,
\end{align}
while using a Gronwall type argument in the $v$-variable, we obtain from \eqref{eq:394}
\begin{align}
  \label{eq:396}
  \left|\frac{\partial^3x}{\partial u^3}(u,v)\right|\leq C+(a+\overline{C}_1v)\left|\frac{\partial^3x}{\partial u^3}(au,u)\right|+\overline{C}_2v|\beta_+'''(u)|+\overline{C}_3\int_{au}^u\left|\frac{\partial^3x}{\partial v^3}(u',u)\right|du'.
\end{align}
Using \eqref{eq:395} in \eqref{eq:396}, the latter becomes
\begin{align}
  \label{eq:397}
  \left|\frac{\partial^3x}{\partial u^3}(u,v)\right|\leq C+(a+\overline{C}_1v)\left|\frac{\partial^3x}{\partial u^3}(au,u)\right|+\overline{C}_2v|\beta_+'''(u)|.
\end{align}
Using now the estimate for $\beta_+'''(v)$ as given in \eqref{eq:379}, we obtain from \eqref{eq:395}, \eqref{eq:397}:
\begin{align}
  \label{eq:398}
  \sup_{T_\varepsilon}\left|\frac{\partial^3x}{\partial v^3}\right|&\leq C+(a+\overline{C}_1\varepsilon)\sup_{T_\varepsilon}\left|\frac{\partial^3x}{\partial u^3}\right|+\overline{C}_2\varepsilon\sup_{T_\varepsilon}\left|\frac{\partial^3x}{\partial v^3}\right|,\\
    \label{eq:399}
    \sup_{T_\varepsilon}\left|\frac{\partial^3x}{\partial u^3}\right|&\leq C+(a+\overline{C}_1\varepsilon)\sup_{T_\varepsilon}\left|\frac{\partial^3x}{\partial u^3}\right|+\overline{C}_2\varepsilon\sup_{T_\varepsilon}\left|\frac{\partial^3x}{\partial v^3}\right|.
\end{align}
Choosing $\varepsilon$ sufficiently small, we deduce from \eqref{eq:398}
\begin{align}
  \label{eq:400}
  \sup_{T_\varepsilon}\left|\frac{\partial^3x}{\partial v^3}\right|\leq C+\overline{C}\sup_{T_\varepsilon}\left|\frac{\partial^3x}{\partial u^3}\right|.
\end{align}
This in \eqref{eq:399} yields
\begin{align}
  \label{eq:401}
  \sup_{T_\varepsilon}\left|\frac{\partial^3x}{\partial u^3}\right|&\leq C+(a+\overline{C}\varepsilon)\sup_{T_\varepsilon}\left|\frac{\partial^3x}{\partial u^3}\right|.
\end{align}
Choosing $\varepsilon$ sufficiently small, we obtain
\begin{align}
  \label{eq:402}
  \sup_{T_\varepsilon}\left|\frac{\partial^3x}{\partial u^3}\right|&\leq C,
\end{align}
which, through \eqref{eq:400}, implies
\begin{align}
  \label{eq:403}
  \sup_{T_\varepsilon}\left|\frac{\partial^3x}{\partial v^3}\right|\leq C.
\end{align}
Using the bounds \eqref{eq:402}, \eqref{eq:403} in \eqref{eq:379} we obtain
\begin{align}
  \label{eq:404}
  \sup_{[0,\varepsilon]}|\beta_+'''|\leq C,
\end{align}
i.e.~we have shown that the third order derivatives of $\beta_+(v)$, $x(u,v)$ are bounded. As mentioned above, by $\beta(u,v)=\beta_+(v)$, $\alpha(u,v)=\beta_+(u)$ and the fact that the derivatives of $t(u,v)$ can, through the characteristic equations, be expressed in terms of the derivatives of $x(u,v)$, all derivatives of third order of $\alpha$, $\beta$, $t$, $x$ are bounded.

We note again that all the smallness conditions on $\varepsilon$ made in the argument above depend only on constants of the type $\overline{C}$, i.e.~do not depend on the order of derivative studied.
\end{proof}

\section*{Acknowledgements}
The author thanks Anne Franzen for many stimulating discussions on the subject and Demetrios Christodoulou for carefully reading the manuskript.

\bibliography{lisibach}
\bibliographystyle{abbrv}

\end{document}

%% file: physical_picture.tex
\begin{tikzpicture}[scale=1.2]
  \fill [fill=gray!20, domain=0:2.5, variable=\x] (0,0) -- plot (\x,{-.4*\x-0.2*\x^2}) -- (0,-2.25) -- cycle;
  \fill [fill=gray!40, domain=0:2.5, variable=\x] (0,0) -- plot (\x,{-.4*\x-0.2*\x^2}) -- ++(1.8,0)-- (4.3,0) -- cycle;
  \fill [fill=gray!40, domain=0:1.82, variable=\x] (0,0) -- plot (\x,{.8*\x+0.2*\x^2}) -- ++(2.5,0)-- (4.3,0) -- cycle;
  \fill [fill=gray!60, domain=0:1.82, variable=\x] (0,0) -- plot (\x,{.8*\x+0.2*\x^2}) -- (0,2.11848) -- cycle;

  \draw[thick,domain=0:1.82,smooth,variable=\x] plot (\x,{.8*\x+0.2*\x^2});
  \draw[thick,domain=0:2.5,smooth,variable=\x] plot (\x,{-.4*\x-0.2*\x^2});
  
  \begin{scope}[yshift=-2.25cm,rotate=90]
    \foreach\n in{0,...,20}
    \draw ({0.2*\n},0)-- ++(.4,.4);
    \draw (0,0.2)--(0.2,.4);
    \draw (4.2,0)--(4.4,0.2);
  \end{scope}
  \draw[thick] (0,-2.25)--(0,2.12);
  
  \draw[-latex] (-1,0) -- (6,0) node[anchor=west] {$x$};
  \draw[-latex] (0,-2.2) -- (0,3) node[anchor=east] {$t$};

  \draw[latex-](1.65,1.8)--(2.2,1.8)node[anchor=west] {reflected shock};
  \draw[latex-](2,-1.5)--(2.3,-1.5)node[anchor=west] {incident shock};
  \node[anchor=north] at (-1.1,1.8) {wall};
  \node[anchor=north] at (-1.1,1.5) {($x=0$)};

  \begin{scope}[xshift=0.5cm,yshift=0.45cm]
    \draw[thick,dashed,domain=0:.208,smooth,variable=\x] plot (\x,{7*\x+5*\x^2});
    \draw[thick,dashed,domain=0:.355,smooth,variable=\x] plot (\x,{-3*\x+1*\x^2});    
  \end{scope}
  \begin{scope}[xshift=0.86cm,yshift=-0.5cm]
    \draw[thick,dashed,domain=0:.243,smooth,variable=\x] plot (\x,{-6*\x-5*\x^2});    
  \end{scope}
  \draw[latex-](.63,.2)--(1.3,.4)node[anchor=west] {particle path};
  
\end{tikzpicture}

%% file: shock_notation.tex
\begin{tikzpicture}[scale=1.2]

  \begin{scope}[yshift=7.87cm]
    \fill [fill=gray!20, domain=.5:1.9, variable=\x] (.5, -7.87) -- plot ({\x}, {.5*\x^2-8}) -- (1.9, -7.87) -- cycle;
    \fill [fill=gray!40, domain=.5:1.9, variable=\x] (.5, -6.2) -- plot ({\x}, {.5*\x^2-8}) -- (1.9, -6.2) -- cycle;
    \fill[fill=gray!20] (1.89,-7.87) rectangle (3.8,-6.2);
    \fill[fill=gray!40] (-1,-7.87) rectangle (.51,-6.2);
    \draw[thick,domain=.5:1.9,smooth,variable=\x] plot (\x,{.5*\x^2-8});
    \draw[dashed,thick](1.3,-7.155)--(1.2,-6.2);
    \draw[dashed,thick] (1.3,-7.155)--(1.2,-7.88);
  \end{scope}
  \node[anchor=west] at (1.75,1.5) {$\mathcal{S}$};
  \node[anchor=west] at (1.8,.5) {state ahead};
  \node[anchor=west] at (-1,1.4) {state behind};
  \begin{scope}[xshift=-1cm]
    \draw[-latex] (-0.2,0) -- (5.3,0) node[anchor=west] {$x$};
    \draw[-latex] (0,-0.2) -- (0,2.1) node[anchor=south] {$t$};
  \end{scope}
  \begin{scope}[xshift=3cm,yshift=2cm]
  \draw[dashed](2,-1.5)--(3,-1.5)node[anchor=west]{Fluid flow line};
  \draw[thick](2,-1)--(3,-1)node[anchor=west]{Shock curve $\mathcal{S}$};
  \end{scope}
\end{tikzpicture}

%% file: future_development.tex
\begin{tikzpicture}[scale=1.3]
  \fill [fill=gray!40, domain=0:.5, variable=\x] (0,0) -- plot (\x,{2*\x+2*\x^2}) -- ++(5,0)-- (5.5,0) -- cycle;

  \draw[thick,domain=0:.5,smooth,variable=\x] plot (\x,{2*\x+2*\x^2});
  
  \draw[-latex] (-1,0) -- (6,0) node[anchor=west] {$x$};
  \draw[-latex] (0,-.2) -- (0,2) node[anchor=east] {$t$};

  \draw[latex-](.5,1.3)--(1,1.3)node[anchor=west] {outgoing charactertistic};
  \node[anchor=north] at (.5,1.9) {$\mathcal{B}$};
  \node[anchor=north] at (4,.8) {future development};
  \draw[latex-](4.5,0)--(4.5,-.5)node[anchor=north] {data};
  \draw[latex-](.5,.45)--(.5,-.5)node[anchor=north] {shock tangent};
  \draw[dashed,thick](0,0)--(.8,.8);
  \fill (0,0) circle[radius=1.2pt];
\end{tikzpicture}

%% file: reflection_problem.tex
\begin{tikzpicture}[scale=1.3]
  \fill [fill=gray!40, domain=0:.68, variable=\x] (0,0) -- plot (\x,{.8*\x+2*\x^2}) -- ++(4.8,0)-- (5.48,0) -- cycle;

  \fill [fill=gray!70, domain=0:.68, variable=\x] (0,0) -- plot (\x,{0.8*\x+2*\x^2}) -- (0,1)-- cycle;
      
  \draw[dashed,domain=0:.5,smooth,variable=\x] plot (\x,{2*\x+2*\x^2});
  \node[anchor=north] at (.5,1.82) {$\mathcal{B}$};
  
  \draw[thick,domain=0:.68,smooth,variable=\x] plot (\x,{.8*\x+2*\x^2});
  \node[anchor=north] at (.75,1.85) {$\mathcal{S}$};
  
  \draw[-latex] (-1,0) -- (6,0) node[anchor=west] {$x$};
  \draw[-latex] (0,-.2) -- (0,2) node[anchor=east] {$t$};


  \node[anchor=north] at (4,.8) {state ahead};
  \draw[latex-](4.5,0)--(4.5,-.5)node[anchor=north] {data};
  \draw[latex-](.5,.45)--(.5,-.5)node[anchor=north] {shock tangent};
  \draw[dashed,thick](0,0)--(.8,.8);
  \fill (0,0) circle[radius=1.2pt];
\end{tikzpicture}

%% file: u-v-coordinates.tex
\begin{tikzpicture}[scale=1.8]

  \begin{scope}[rotate=45]
    \draw[-latex](0,0)--(2,0)node[anchor=south west]{$v$};
    \draw[-latex](0,0)--(0,2)node[anchor=south east]{$u$};
    \draw(0,0)--(1.5,1.5)node[anchor=south]{$u=v$};
    \draw(0,0)--(1.8,.9)node[anchor=south]{$u=av$};
    \draw[fill=gray!40](0,0)--(1.4,.7)--(0.7,0.7)--cycle;
    \node at(0,.7)[anchor=north east]{$a\varepsilon$};
    \draw[dashed](0,.7)--(.7,.7);
    \draw (-.05,.7)--(.05,.7);
    \draw[dashed](1.4,.7)--(1.4,0);
    \draw (1.4,.05)--(1.4,-.05);
    \node at(1.4,0)[anchor=north west]{$\varepsilon$};
    \node at(.71,.5){$T_\varepsilon$};
  \end{scope}

\end{tikzpicture}

%% file: delxdelu.tex
\begin{tikzpicture}[scale=1.8]

  \begin{scope}[rotate=45]
    \draw[-latex](1.15,0)--(2.2,0)node[anchor=south west]{$v$};
    \draw(0,0)--(.9,0);
    \draw[-latex](0,0)--(0,2.2)node[anchor=south east]{$u$};
    \draw(0,0)--(1.6,1.6)node[anchor=south]{$u=v$};
    \draw(0,0)--(2.1,1.05)node[anchor=south]{$u=av$};
    \draw[fill=gray!40](0,0)--(2,1)--(1,1)--cycle;
    \draw[-latex](0.7,0.7)--(.9,0.7);
    \draw(0.7,0.7)--(1,0.7);
    \draw[-latex](0.7,0.35)--(0.7,0.58);
    \draw(0.7,0.35)--(0.7,0.7);
    \node at (1,0.7)[anchor=south]{\small $(u,v)$};
    \node at (0.7,0.7)[anchor=east]{\small$(u,u)$};
    \node at (0.7,.35)[anchor=west]{\small$(au,u)$};
    \fill (1,0.7) circle[radius=.7pt];
    \fill (.7,0.7) circle[radius=.7pt];
    \fill (0.7,.35)circle[radius=.7pt];
  \end{scope}

  \begin{scope}[xshift=4cm,rotate=45]
    \draw[-latex](0,0)--(2.2,0)node[anchor=south west]{$v$};
    \draw[-latex](0,0)--(0,2.2)node[anchor=south east]{$u$};
    \draw(0,0)--(1.6,1.6)node[anchor=south]{$u=v$};
    \draw(0,0)--(2.1,1.05)node[anchor=south]{$u=av$};
    \draw[fill=gray!40](0,0)--(2,1)--(1,1)--cycle;
    \draw[-latex](1,.5)--(1,.65);
    \draw(1,.5)--(1,0.7);
    \node at (1,0.7)[anchor=south]{\small $(u,v)$};
    \node at (1,.5)[anchor=west]{\small$(av,v)$};
    \fill (1,0.7) circle[radius=.7pt];
    \fill (1,0.5) circle[radius=.7pt];
  \end{scope}

\end{tikzpicture}